\documentclass[11pt, psamsfonts]{amsart}
\usepackage{amsmath}
\usepackage{amsthm}
\usepackage{amssymb}
\usepackage{amscd}
\usepackage{amsfonts}
\usepackage{amsbsy}
\usepackage{stmaryrd}
\usepackage{epsfig}
\usepackage{pdfsync, hyperref,epstopdf}

\newtheorem{remark}{Remark}
\newtheorem{proposition}{Proposition}
\newtheorem{lemma}[proposition]{Lemma}

\newtheorem*{claim*}{Claim}
\newtheorem{definition}{Definition}

\def\ie{{\em i.e.,\ }}
\def\eg{{\em e.g.\ }}

\def\eps{\varepsilon}

\def\R{{\mathbb R}}

\def\N{{\mathbb N}}

\def\le{\leqslant}
\def\ge{\geqslant}

\newdimen\AAdi%
\newbox\AAbo%
\def\AArm{\fam0 }
\def\AAk#1#2{\setbox\AAbo=\hbox{#2}\AAdi=\wd\AAbo\kern#1\AAdi{}}%
\def\AAr#1#2#3{\setbox\AAbo=\hbox{#2}\AAdi=\ht\AAbo\raise#1\AAdi\hbox{#3}}%

\def\BBone{{\AArm 1\AAk{-.8}{I}I}}%

\newcommand {\CA}{{\mathcal A}}

\newcommand {\CC}{{\mathcal C}}

\newcommand {\CL}{{\mathcal L}}

\newcommand {\CP}{{\mathcal P}}

\newcommand{\disp}{\displaystyle}
\newcommand{\8}{\infty}

\def\m1{{-1}}

\newcommand{\ninf}{{n\rightarrow\8}}
\newcommand{\ol}{\overline}
\def\S{\Sigma}
\def\s{\sigma}

\newcommand{\wh}{\widehat}
\newcommand{\wt}{\widetilde}
\newcommand{\al}{\alpha}
\newcommand{\be}{\beta}

\newcommand{\0}{0^{\8}}

\newcommand{\1}{1^{\8}}

\newcommand{\uij}{u_{ij}^{\8}}
\newcommand{\frth}{\frac\theta{1-\theta}}

\newcommand{\ul}[1]{\underline{#1}}

\begin{document}

\title[Selection  of maximizing measure]
{Flatness is a criterion for selection of maximizing measures}
\author{Renaud Leplaideur}
\date{Version of \today}
\thanks{Part of this research was supported by DynEurBraz IRSES FP7 230844.
}

\maketitle

\begin{abstract}
For a full shift with $Np+1$ symbols and for a non-positive potential, locally proportional to the distance to one of  $N$  disjoint full shifts with  $p$ symbols, we prove  that the equilibrium state converges as the temperature goes to 0. 

The main result is that the limit is a convex combination of the two ergodic measures with maximal entropy among maximizing measures and whose supports are the two shifts where the potential  is the flattest. 

In particular, this is a hint to solve the open problem of selection, and this indicates that flatness is probably a/the criterion for selection as it was conjectured by A.O. Lopes.  

As a by product we get convergence of the eigenfunction at the log-scale to a unique calibrated subaction. 
\end{abstract}

\section{Introduction}
\subsection{Background}
In this paper we deal with the problem of ergodic optimization and, more precisely with the study of grounds states. Ergodic optimization is a relatively new branch of ergodic theory and is very active since the 2000's. For a given dynamical system $(X,T)$, the goal is to study existence and properties of the invariant probabilities which maximize a given potential $\phi:X\to \R$. We refer the reader to \cite{jenkinson-survey} for a survey about ergodic optimization. 

Ground states are particular maximizing measures which can be reached by freezing the system as limit of equilibrium states. Namely, for $\beta>0$, which in statistical mechanics represents the inverse of the   temperature, we consider the/an equilibrium state associated to $\be \phi$, that is a $T$-invariant probability whose free energy 
$$h_{\nu}(T)+\be\int\phi\,d\nu,$$
is maximal
(where $h_{\nu}$ is the Kolmogorov entropy of the measure $\nu$). 
Then, considering an equilibrium state $\mu_{\be}$, it is easy to check (see \cite{CLT}) that any accumulation point for $\mu_{\be}$  as $\be$ goes\footnote{$\beta$ is the inverse of the temperature, thus $\be\to+\8$ means $Temp\to 0$, which means we freeze the system.} to $+\8$ is a maximizing measure for $\phi$. 

The first main question is to know if $\mu_{\be}$ converges. It is known (see \cite{bousch-walters}) that for an uniformly hyperbolic dynamical systems, generically in the $\CC^{0}$-topology, $\phi$ has a unique maximizing measure. Therefore, convergence of $\mu_{\be}$ obviously holds in that case. 

Nevertheless, generic results do not concern all the possibilities, and it is very easy to build examples, which at least for the mathematical point of view are meaningful, and for which the set of ergodic maximizing measures is as wild as wanted. 

For these situations, the question of convergence is of course fully relevant. Cases of convergence or non-convergence are known (see \cite{bremont, leplaideur-max, chazottes-gambodau-ulgade, chazottes-hochman}), but the general theory is far away of being solved. In particular, no criterion which guaranties convergence (except the uniqueness of the maximizing measure or the locally constant case) is known, say \eg for the Lipschitz continuous case. 

The second main question, and that is the one we want to focus on here, is the problem of the selection. Assuming that $\phi$ has several ergodic maximizing measures and that $\mu_{\be}$ converges, what is the limit. In other words, is there a way to predict the limit from the potential, or equivalently, what makes the equilibrium state select one locus instead of another one ?  

Inspired by a similar study for the Lagrangian Mechanics (\cite{anantharaman-flatness}), it was conjectured by A.O. Lopes that flatness of the potential would be a criterion for selection and that the equilibrium state always selects the locus where the potential is the flattest.
In \cite{aar-selection} it is actually proved that the conjecture is not entirely correct. Authors consider  in the full 3-shift a negative potential except on the two fixed points $\0$ and $\1$,  where  it vanishes but is sharper in $\1$  than in $\0$.  Then, they prove that the equilibrium state actually converges, but not necessarily to the Dirac measure at $\0$. 

The first part of the conjecture is however not (yet) invalided and the question to know if flatness is a criterion for selection is still relevant. 

Here, we make a step in the direction of proving  that flatness is indeed a criterion for selection. Precise statements are given in the next subsection. We consider in the full shift with $Np+1$ symbols a potential, negative everywhere except on $N$ Bernoulli subshifts\footnote{Note that these Bernoulli shifts are empty interior compact sets.} with $p$ symbols. Figure \ref{fig-shift1} illustrates the dynamics.

\begin{figure}[htbp]
\includegraphics[scale=0.5]{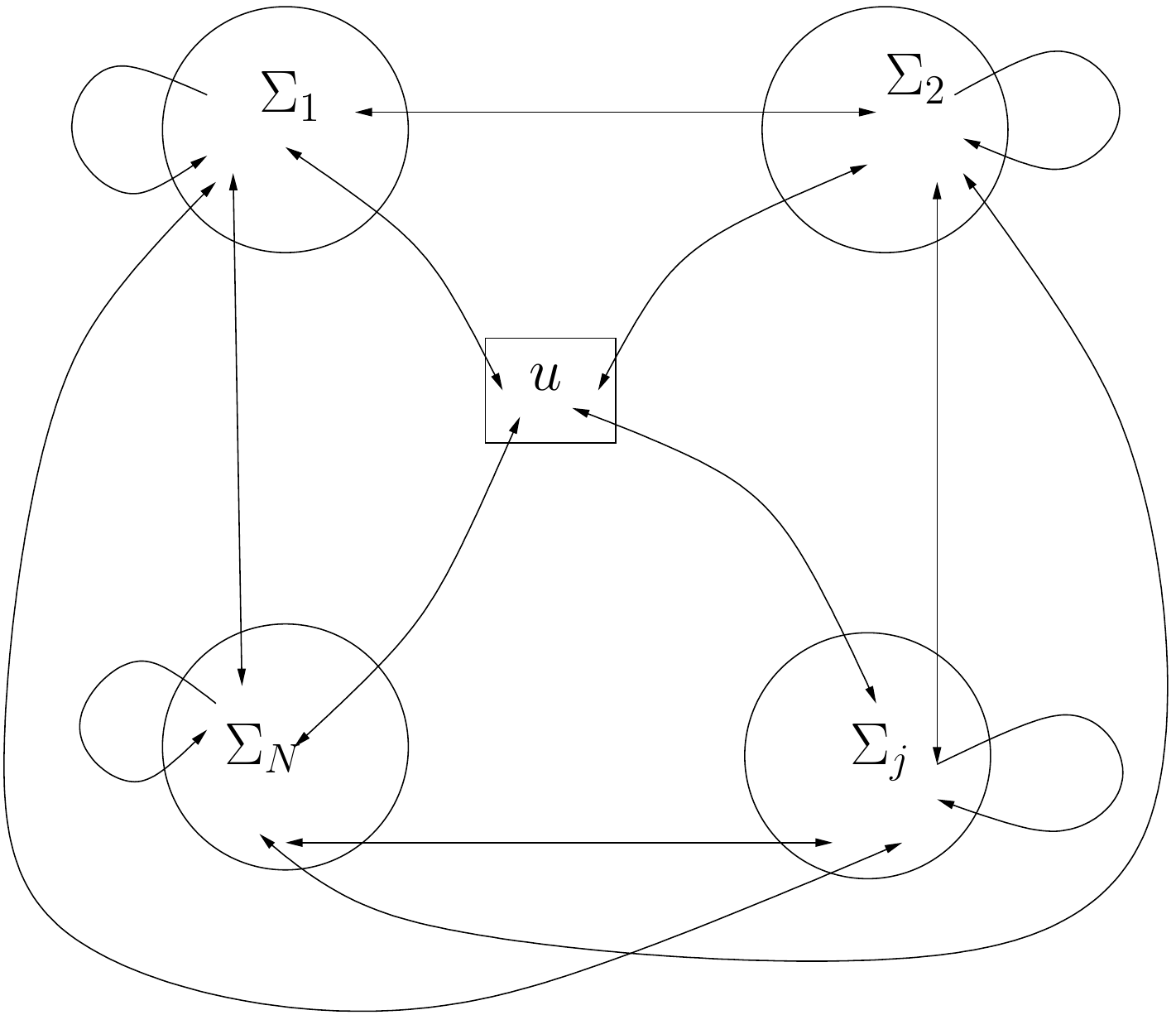}
\caption{Our system}
\label{fig-shift1}
\end{figure}

Then, flatness is ordered on these $N$ Bernoulli subshifts : the potential is flatter on the first one than on the second one,  then, it is flatter on the second one than on the third one, and then so on (see below for complete settings). 

Any of these Bernoulli shifts has a unique measure of maximal entropy, and the set of ground states is contained in the convex hull of these $N$-measures.  
We show here that the equilibrium state converges and selects a convex combination of the two ergodic measures with supports in the two flattest Bernoulli subshifts. 

We emphasize that this result is absolutely not in contradiction with \cite{aar-selection}. Indeed, in \cite{aar-selection} it is proved that the equilibrium state converges to a convex combination of the Dirac measures at $\0$ and $\1$, which are obviously the two flattest loci !

\subsection{Settings}
\subsubsection{The set $\S$}
We consider the full-shift $\S$ with $Np+1$ symbols, with $N$ and $p$ two positive integers. We also consider $Np+1$ positive real numbers, 
$0<\al_{1}<\al_{2}\le \ldots\le \al_{p}$, $0<\al_{p+1}<\al_{p+2}\le \ldots\le \al_{2p}$, $\ldots, 0<\al_{(N-1)p+1}<\al_{(N-1)p+2}\le \ldots\le \al_{Np}$ and $\al$. 
We assume 
$$\al_{1}<\al_{p+1}<\al_{2p+1}\le \al_{(N-1)p+1}.$$
We set $\S_{1}:=\{1,\ldots, p\}^{\N}$, $\S_{2}:=\{p+1,\ldots, 2p\}^{\N}$, $\ldots, \S_{N}:=\{(N-1)p+1,\ldots, Np\}^{\N}$.

 For simplicity the last letter $Np+1$ is denoted by $u$. The letter $(j-1)p+i$ will be denoted by $u_{ij}$.

The set of letters defining $\S_{j}$ is $\CA_{j}:=\{u_{ij}, 1\le i\le p\}$. Hence, $\S_{j}=\CA_{j}^{\N}$, and a word  admissible for $\S_{j}$ is a word (finite or infinite) in letters $u_{ij}$. The length of a word is the number of digit (or letter) it contains. The length of the word $w$ is denoted by $|w|$. 

If $w=w_{0}w_{1}\ldots w_{n}$ and $w'=w'_{0},w'_{1},\ldots w'_{n'}$ are two finite words, we define the concatenated word $ww'=w_{0}w_{1}\ldots w_{n}w'_{0},w'_{1},\ldots w'_{n'}$. This is easily extended to the case $|w'|=+\8$. If $m$ is a finite-length admissible word for $\S_{j}$, $[m*]$ denotes the set of points starting with the same $|m|$ letters than $m$ and whose next letter is not in $\CA_{j}$.

The distance in $\S$ is defined (as usually) by 
$$d(x,y)=\theta^{\min\{j,\ x_{i}\neq y_{i}\}},$$
where $\theta$ is a fixed real number in $(0,1)$. 
This distance is sometimes graphically represented as in Figure \ref{fig:distance}.
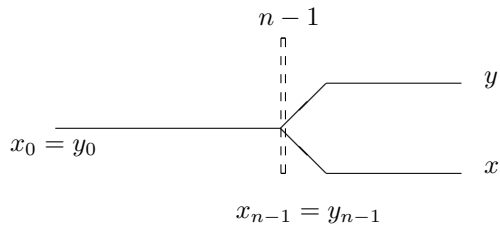
\begin{figure}[ht]
\unitlength=6mm
\begin{picture}(12,4.5)(-2,0)
\put(-1,1.5){\small$x_{0}=y_{0}$}
\put(0,2){\line(1,0){5}}
\put(5,2){\line(1,1){1}} \put(5,2){\line(1,-1){1}}
\put(5,1){\dashbox{0.2}(0.1,3)}\put(4.5,4.3){\small $n-1$} 
\put(4,0){\small $x_{n-1}=y_{n-1}$}
\put(6,3){\line(1,0){3}}
\put(9.5,3){\small $y$}
\put(6,1){\line(1,0){3}}
  \put(9.5,1){\small $x$} 
\end{picture}
\caption{The sequence $x$ and $y$ coincide for digits $0$ up to $n-1$ and then split.}\label{fig:distance}
\end{figure}

We emphasize here, that contrarily to \cite{aar-selection} we have not chosen $\theta=\frac12$ in view to get the most general result as possible. Indeed, in \cite{aar-selection} it was not clear if some results where independent or not of $\theta$'s value. 
Moreover, this also means that we are considering all H\"older continuous functions, and not only the Lipschitz ones, because in $\S$, a H\"older continuous function can be considered as a Lipschitz continuous, up to a change of $\theta$'s value.

\subsubsection{The potential, the Transfer operator and the Gibbs measures}
The potential $A$ is defined by 
$$\disp A(x)=
\left\{
\begin{array}{l}
-\alpha_{j}d(x,\S_{i}), \text{ if }x\in[u_{ij}]\\
A(x)=-\al, \text{ if }x\in[u].\\
\end{array}\right.$$
The potential is negative on $\S$ but on each $\S_{j}$ where it is constant to 0.

The transfer operator, also called the Ruelle-Perron-Frobenius operator, is defined by 
$$\CL_{\beta}(\varphi)(x):=\sum_{j=1}^{k}\sum_{i=1}^{p}e^{A(u_{ij}x)}\varphi(u_{ij}x)+e^{A(ux)}\varphi(ux).$$
It acts on continuous functions $\varphi$. We refer the reader to Bowen's book \cite{bowen} for detailed theory of transfer operator, Gibbs measures and equilibrium states for Lipschitz potentials. 

The eigenfunction is $H_{\beta}$ and the eigenmeasure is $\nu_{\beta}$. They satisfy:
$$\CL_{\beta}(H_{\beta})=e^{\CP(\be)}H_{\beta},\quad \CL_{\beta}^*(\nu_{\beta})=e^{\CP(\be)}\nu_{\beta}.$$
The eigenmeasure and the eigenfunction are uniquely determined if we required the assumption that $\nu_{\be}$ is a probability measure and $\disp \int H_{\be}\,d\nu_{\be}=1$.

\emph{The Gibbs state} $\mu_{\be}$ is defined by $d\mu_{\be}:=H_{\be}d\nu_{\be}$. The measure $\mu_{\be}$ is also the equilibrium state for the potential $\be.A$ : it satisfies 
$$\max_{\mu\ \s-inv}\left\{h_{\mu}(\s)+\be\int A\,d\mu\right\}=h_{\mu_{\be}}(\s)+\be\int A\,d\mu_{\be}$$
and this maximum is $\CP(\be)$ and is called the \emph{pressure} of $\be A$. $e^{\CP(\be)}$ is also the spectral radius of $\CL_{\be}$. It is a single dominating eigenvalue.

\subsection{Results}
In each $\S_{j}$ we get a measure of maximal entropy $\mu_{top,j}$. As each $\S_{j}$ is a subshift of finite type, $\mu_{top,j}$ is again of the form 
$$d\mu_{top,j}=H_{top,j}d\nu_{top,j},$$
where $H_{top,j}$ and $\nu_{top,j}$ are respectively the eigenfunction and the eigen-probability associated to the transfer operator in $\S_{j}$ for the potential constant to 0.

Note that, as $\be$ goes to $+\8$, $\mu_{\be}$ has only $N$ possible ergodic accumulation points, which are the measures of maximal entropy in each $\S_{j}$, $\mu_{top,j}$. 

Our results are 

\medskip\noindent{\bf Theorem 1}
{\it The eigenmeasure $\nu_{\be}$ converges to  the eigenmeasure $\nu_{top,1}$ for the weak* topology as $\be$ goes to $+\8$.}
\medskip\noindent

\medskip\noindent{\bf Theorem 2}
{\it The Gibbs measure $\mu_{\be}$ converges to  a convex combination of $\mu_{top,1}$ and $\mu_{top,2}$ for the weak* topology as $\be$ goes to $+\8$. This combination depends in which zone $Z_{1}\cup Z_{2}\cup Z_{3}\cup Z_{4}$ (see Figure \ref{fig-mu}) the parameters are:
\begin{enumerate}
\item For parameters in $Z_{1}$, $\mu_{\be}$ converges to $\disp\frac1{1+p^{2}}(\mu_{top,1}+p^{2}\mu_{top,2})$.

\item For parameters in $Z_{2}$, $\mu_{\be}$ converges to $\mu_{top,1}$.

\item  For parameters in $Z_{3}\cup Z_{4}$ $\mu_{\be}$ converges to $\disp\frac1{1+\rho_{i}^{2}}(\mu_{top,1}+\rho_{i}^{2}\mu_{top,2})$ for some $\rho_{i}>0$  locally constant in $Z_{3}\setminus Z_{4}$, $Z_{4}\setminus Z_{3}$ and $Z_{3}\cap Z_{4}$(see Equalities \eqref{equ-def-rho-3} page \pageref{equ-def-rho-3}, \eqref{equ-def-rho-4} page \pageref{equ-def-rho-4} and \eqref{equ-def-rho-34} page \pageref{equ-def-rho-34}).
\end{enumerate}
}
\medskip\noindent

\begin{figure}[htbp]
\begin{center}
\includegraphics[scale=0.6]{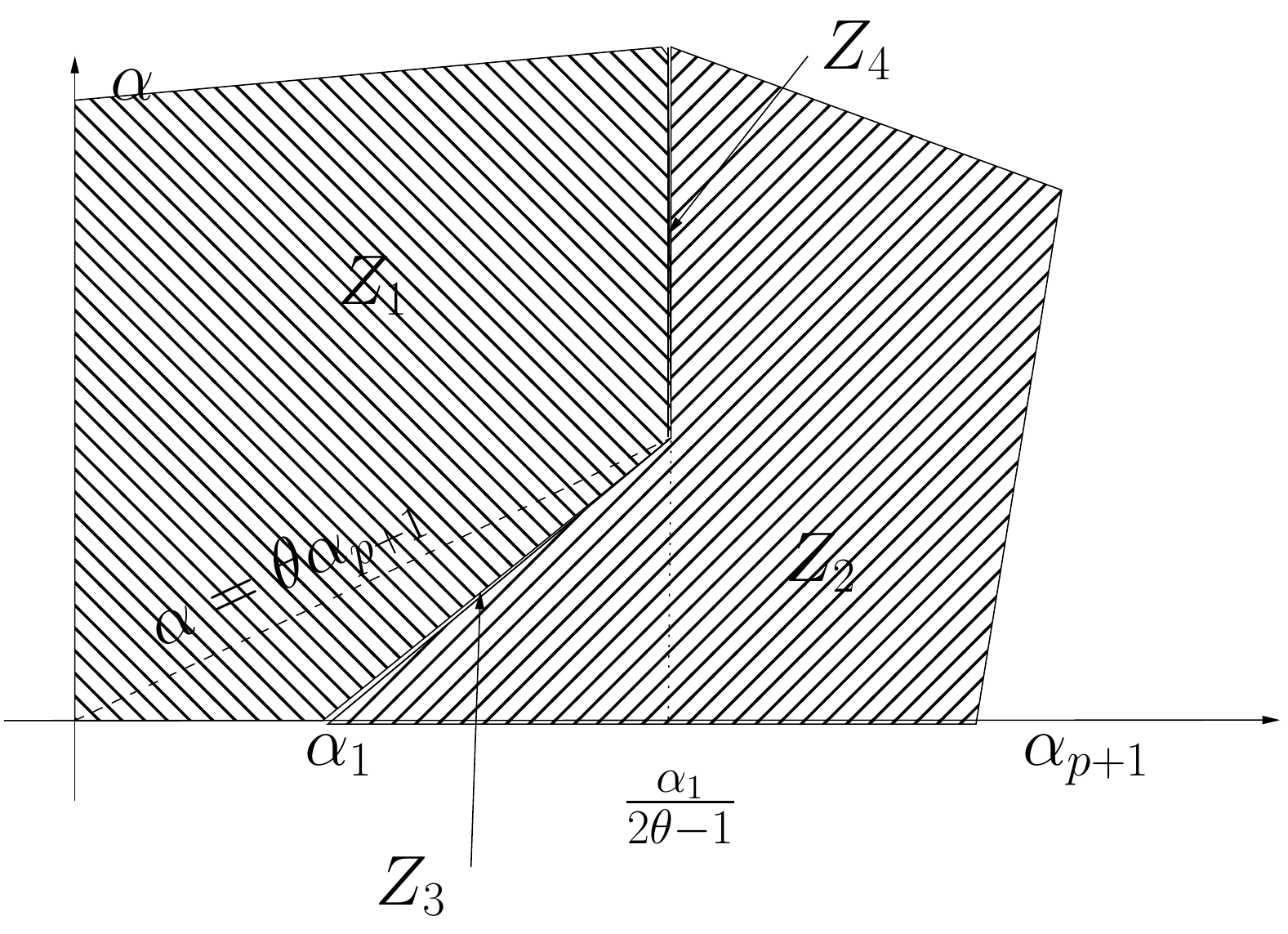}
\caption{Ratio between $\mu_{top,1}$ and $\mu_{top,2}$}
\label{fig-mu}
\end{center}
\end{figure}

Zone $Z_{3}$ corresponds to $0<\al\le\al_{p+1}\theta$ and $\al=\frth\frac{\al_{p+1}-\al_{1}}2$. Zone $Z_{4}$ corresponds to $\al_{p+1}=\disp\frac{\al_{1}}{2\theta-1}$ and $\al\ge \al_{p+1}\theta$. We emphasize that $Z_{4}$ exists if and only if $\theta>\frac12$.

As a by product of Theorem 1 and Theorem 2 we get the exact convergence  for the eigenfunction to a unique subaction (see Section \ref{sec-peierls} for definition):

\medskip\noindent{\bf Corollary 3}
{\it The calibrated subactions are all equal up to an additive constant. Moreover, the eigenfunction $H_{\be}$ converges at the log-scale to a single calibrated subaction :
$$V:=\lim_{\be\to+\8}\frac1\be\log(H_{\be}).$$
}
\medskip\noindent

The question of convergence and uniqueness of a subaction seems to be important for the theory of ergodic optimization. 
It appeared very recently  in \cite{bara-lopes-mengue}. We point out that convergence of the eigenfunction to a subaction is related to the study of a  Large Deviation Principle  for the convergence of $\mu_{\be}$ (see \eg \cite{bara-lopes-thieu} and \cite{lopes-oliveira-smania}). Nevertheless, in these two papers, Lopes {\it et al.} always assume the uniqueness of the maximizing measure, which yields the uniqueness of the calibrated subaction (up to a constant). Here we prove convergence to a unique subaction  without the assumption of the uniqueness of the maximizing measure. 
It it thus allowed to hope that we could get a more direct proof of a Large Deviation Principle, without using the very indirect machinery of dual shift (see \cite{bara-lopes-thieu}).

\subsection{Further improvements: discussion on hypothesis}

The present work is part of a work in progress. The situation described here is far away from the most general case and our goal is to prove the next conjecture. 

In \cite{garibaldi-lopes-thieullen-09}, Garibaldi \emph{et al.} introduce the set of non-wandering points with respect to a H\"older continuous potential $A$, $\Omega(A)$. This set contains the union of the supports of all optimizing measures (here we consider maximizing measures, there they consider minimizing measures). 

The set $\Omega(A)$ is invariant and compact. Under the assumption that it can be decomposed in finitely many irreducible pieces, it is shown that calibrated subactions are constant on these irreducible pieces and their global value is given by these local values and the Peierls barrier (see Section \ref{sec-peierls} below).

We believe that, under the same hypothesis, it is possible to determine which irreducible component have measure at temperature zero:

\medskip\noindent{\bf Conjecture}.
{\it For $A:\S\to \R$ H\"older continuous, if $\Omega(A)$ has finitely many irreducible components, $\Omega(A)_{1}\ldots \Omega(A)_{N}$, then, $\mu_{\be}(\Omega(A)_{i})$ goes to $0$ if $\Omega(A)_{i}$ is not one of the two flattest loci for $A$.}
\medskip\noindent

We emphasize that this conjecture does not mean that there is convergence ``into'' the components. It may be (as in \cite{chazottes-hochman}) that an irreducible components has several maximizing measures and that there is no selection between these measures. 

This conjecture is for the moment far of being proved, in particular because several notions are not yet completely clear. In particular the notion of flatness has to be specified. Moreover, the components are not necessarily subshifts of finite type, which is an obstacle to study their (for instance) measures of maximal entropy. 

The work presented here, is for a specific form of potential for which flatness is easily defined. The dynamics into the irreducible components and also the global dynamics are easy. We believe that the main issue here  is to identify  flatness as a criterion for selection.

The next step would be to release assumptions on the dynamics; in particular we would like that theses components are not full shift  and that the global dynamics is not a full shift. It is also highly  probable that the conjecture should be adapted after we have solve this case. Distortion into the dynamics could perhaps favor other components. 

The last step would be to get the result for general (or as general as possible) potential.

Nevertheless, and even if the present work is presented as a work in progress and an intermediate step before a more general statement, we want to moderate the specificity of the potential we consider here. For a uniformly hyperbolic system $(X,T)$ and for any H\"older function  $\phi$, there exists two H\"older continuous $\psi_{1}$ and $\psi_{2}$ such that $\phi=\psi_{1}+m(\phi)+\psi_{2}\circ T-\psi_{2}$, where $\disp m(\phi)=\max\left\{\int \phi\,d\mu\right\}$ and $\psi_{1}$ is non-negative and vanishes only on the Aubry set. 
This means that, up to consider a cohomologous function, the assumption on the sign  of $A$   is free. Now, if we would consider a very regular potential (say at least $\CC^{1}$) on a geometrical dynamical systems, the fact that $\psi_{1}$ vanishes on the Aubry set means that close to that Aubry set, $\psi_{1}(x)$ is proportional to the distance between $x$ and the Aubry set (with coefficient related to the derivative of $\psi_{1}$).  Consequently, the potential $A$ we consider here is a kind of discrete version for the symbolic case of a $\CC^{1}$ potential on a Manifold.

\subsection{Plan of the paper and acknowledgment}
In Section \ref{sec-peierls} we prove that the pressure behaves like $\log p+ g(\be)e^{-\gamma\be}$ for some specific $\gamma$ and some sub-exponential function $g$. The real number $\gamma$ is obtained as an eigenvalue for the Max-Plus algebra (see Proposition \ref{prop2-maxplus}).

In Section \ref{sec-F}, we define and study an auxiliary function $F$; This function gives the asymptotic both for the eigenmeasure and for the Gibbs measure. 

In Section \ref{sec-proof-th1} we prove Theorem 1 and in Section \ref{sec-proof-th2} we prove Theorem 2. As a by-product we give an asymptotic for the function $g(\be)$ (in the Pressure). 

In the last section, Section \ref{sec-proof-coro} we prove Corollary 3.

\bigskip
Part of this work was done as I was visiting E. Garibaldi at Campinas (Brazil). I would like to thank him a lot here for the talks we get together and the attention he gave me to listen and correct some of my computations.


\section{Peierls barrier and an expression for the pressure}\label{sec-peierls}
The main goal of this section is Proposition \ref{prop2-maxplus} where we prove that $\CP(\be)$ converges exponentially fast to $\log p$. The exponential speed is obtained as an eigenvalue for a matrix in Max-Plus algebra. 
 
 \subsection{The eigenfunction and the Peierls barrier}\label{subsec-peierls}

\begin{lemma}\label{lem2-aneis}
The eigenfunction $H_{\be}$ is constant on  the cylinder $[u]$. It is also constant on cylinders of the form $[m*]$, where $m$ is an admissible word for some $\S_{j}$. Furthermore, if $m'$  is another admissible word for the same $\S_{j}$ with the same length than $m$, then $H_{\be}$ coincide on both cylinders $[m*]$ and $[m'*]$.
\end{lemma}
\begin{proof}
The eigenfunction $H_{\be}$ is defined by 
$$H_{\be}:=\lim_{\ninf}\frac1n\sum_{k=0}^{n-1}e^{-k\CP(\be)}\CL_{\be}(\BBone).$$
For $x$ and $x'$ in $[m*]\cup[m'*]$, if $wx$ is a preimage for $x$ then $wx'$ is a preimage for $x'$ and $A(wx)=A(wx')$. 

The same argument works on $[u]$. 
\end{proof}

This Lemma allows us to set 
 \begin{equation}
\label{equ-def-tauj}
\tau_{j}:=\sum_{l\neq j}\sum_{i}e^{-\al_{(l-1)p+i}\theta\beta}H_{\be}(u_{il}*)+e^{-\al}H_{\be}(u).
\end{equation}
Thus, for $x$  in $\sqcup_{i}[u_{ij}]$ we get
$$e^{\CP(\be)}H_{\be}(x)=\sum_{i=1}^{p}e^{\be A(u_{ij}x)}H_{\be}(u_{ij}x)+\tau_{j}.$$

\begin{lemma}
\label{lem-hb-const-aubry}
The function $H_{\be}$ is constant on each $\S_{j}$: for any $x$ in $\S_{j}$, 
$H_{\be}(x)=\frac{\disp\tau_{j}}{\disp e^{\CP(\be)}-p}.$

\end{lemma}
\begin{proof}
The function $H_{\be}$ is continuous on the compact $\S_{j}$. It thus attains its minimum and its maximum. Let $\ul x_{j}$ and $\ol x_{j}$ be two points in $\S_{j}$ where $H_{\be}$ is respectively minimal and maximal. 

The transfer operator yields:
$$e^{\CP(\be)}H_{\be}(\ul x_{j})=\left(\sum_{i=1}^{p}H_{\be}(u_{ij}\ul x_{j})\right)+\tau_{j}.$$
By definition, for each $i$, $H_{\be}(u_{ij}\ul x_{j})\ge H_{\be}(\ul x_{j})$. This yields 
\begin{equation}
\label{equ1-deftauj}
(e^{\CP(\be)}-p)H_{\be}(\ul x_{j})\ge \tau_{j}.
\end{equation}
Similarly we will get $\tau_{j}\ge (e^{\CP(\be)}-p)H_{\be}(\ol x_{j})$. As the potential is Lipschitz continuous, the pressure function $\CP(\be)$ is analytic and decreasing ($A$ is non-positive). Then $\CP(\be)>\log p$. 
This shows $H_{\be}(\ul x_{j})=H_{\be}(\ol x_{j})$. 

 Let $x$ be any point in $\S_{j}$. Equality $\CL_{\be}(H_{\be})=e^{\CP(\be)}H_{\be}$ yields 
$$e^{\CP(\be)}H_{\be}(x)=\sum_{i\in \CA_{j}}e^{\be.A(ix)}H_{\be}(ix)+\tau_{j}.$$
For $i\in \CA_{j}$, $A(ix)=0$, and as $H_{\be}$ is constant on $\S_{j}$ we get 
$$(e^{\CP(\be)}-p)H_{\be}(x)=\tau_{j}.$$
\end{proof}

The  family  of functions $\displaystyle\{\frac{1}{\beta} \log H_{\beta}\}_{\be\in \R^{+}}$  is uniformly bounded and equi-continuous; any accumulation point  $V$ for
$\displaystyle\frac{1}{\beta} \log H_{\beta}
$ as $\beta$ goes to $+\8$ (and for the $\CC^{0}$-norm) is a \emph{calibrated subaction}, see \cite{CLT}. 

In the following, we consider a calibrated subaction $V$ obtained as an accumulation point for $\disp\frac1\beta\log H_{\beta}$. Note that the convergence is uniform  (on $\S$) along the chosen subsequence for $\be$. For simplicity we shall however write $V=\lim_{\be\to+\8}\frac1\beta\log H_{\beta}$.

A direct consequence of Lemma \ref{lem-hb-const-aubry} is that the subaction $V$ is constant on each $\S_{j}$. Actually, it is proved in  \cite{garibaldi-lopes} that this holds for the more general case and, moreover, that any calibrated subaction satisfies
\begin{equation}
\label{equ1-suba-cal2}
V(x)=\max_{j}\{V(x_{j})+h(x_{j},x)\},
\end{equation}
where $h(\cdot,\cdot)$ is the \emph{Peierls barrier} and $x_{j}$ is any point in $\S_{j}$. It is thus important to compute what is the Peierls barrier here. 

\begin{lemma}\label{lem2-pierlsbar}
Let $x_{j}$ be any point in $\S_{j}$. 
The Peierls barrier satisfies $h(x_{j},x)=-\al_{(j-1)p+1}\frth d(x,\S_{j})$. 
\end{lemma}
For simplicity we shall set $h_{j}(x)$ for $h(x_{j},x)$.
\begin{proof}
Let $x$ be in $\S$. Recall that $h_{j}(x)$ is defined by 
$$\lim_{\eps\to0}\sup_{n} \left\{\sum_{l=0}^{n-1}A(\s^{l}(z)) : \ \s^{n}(z)=x,\ d(x_{j},z)<\eps\right\}.$$
As we consider the limit as $\eps$ goes to 0, we can assume that $\eps<\theta$. 
Now, to compute $h_{j}(x)$, we are looking for a preimage of $x$, which starts by some letter admissible for $\S_{j}$ (because $\eps<\theta$) and which maximizes the Birkhoff sum of the potential ``until $x$''. As the potential is negative, this can be done if and only if one takes a preimage of $x$ of the form $mx$, with $m$ a $\S_{j}$ admissible word. Moreover, we always have to chose the letter $u_{1j}$ to get the maximal $-\al_{(j-1)p+i}$ possible. 

In other word we claim that for every $n\ge 1$ for every $l\ge 0$ and for every word $m$ of length $n+l$, 
$$S_{n}(A)(\underbrace{u_{1j}\ldots u_{1j}}_{n \text{ times}}x)\ge S_{n+l}(A)(mx).$$
Let assume $\disp d(x,\S_{j})=\theta^{a}$, \ie the maximal admissible word for $\S_{j}$ of the form $x_{0}x_{1}x_{2}\ldots$ has length $a$. 
This yields
\hskip -0.5cm$$h_{j}(x)=\lim_{\ninf}S_{n}(A)(\underbrace{u_{1j}\ldots u_{1j}}_{n \text{ times}}x)=-\al_{(j-1)p+1}\sum_{n=1}^{+\8}\theta^{n+a}=-\al_{(j-1)p+1}\frth\theta^{a}=-\al_{(j-1)p+1}\frth d(x,\S_{j}).$$
\end{proof}

Then, Lemma \ref{lem2-pierlsbar} and Equality \eqref{equ1-suba-cal2} yield

\begin{equation}
\label{equ2-suba-cal2}
V(x)=\max_{j}\left(V(\ol x_{j})-\al_{(j-1)p+1}\frac{\theta}{1-\theta}d(x,\S_{j})\right).
\end{equation}

\begin{lemma}\label{lem2-subac-1234}
For every $j$,   
$$\lim_{\be\to+\8}\frac1\be\log\left(e^{-\alpha_{(j-1)p+1}\theta\beta}H_\beta(u_{1j}*)+\ldots+e^{-\alpha_{jp}\theta\beta}H_\beta(u_{pj}*)\right)=V(u_{1j}*)-\alpha_{(j-1)p+1}\theta.$$
\end{lemma}
\begin{proof}
By Lemma \ref{lem2-aneis}, the eigenfunction $H_{\be}$ is constant on rings $[m*]$ with $m\in \CA_{j}^{|m|}$, hence 
$V(u_{1j}*)=\ldots=V(u_{pj}*)$. Now,  inequalities $\al_{(j-1)p+1}<\al_{(j-1p+)2}\le \al_{(j-1)p+i}$ show that $\disp e^{-\al_{(j-1)p+1}\theta\be}H_\beta(u_{1j}*)$ is exponentially bigger than all the other terms as $\be$ goes to $+\8$. 
\end{proof}

\subsection{Exponential speed of convergence of the pressure : Max-Plus formalism}
Here we use the Max-Plus formalism. We refer the reader to \cite{BCOQ} (in particular chapter 3) for basic notions on this theory.  Some of the results we shall use here are not direct consequence  of \cite{BCOQ} (even if the proofs can easily be adapted) but can be found in \cite{akian-gauber98}.

\begin{proposition}\label{prop2-maxplus}
Let $$\disp\gamma=\min\left\{\min(\al_{p+1}\theta,\al)+\al_{1}\frac{\theta}{1-\theta},
\frac{(\al_{1}+\al_{p+1})}2\frac\theta{1-\theta}\right\}.$$
Then, there exists a positive sub-exponential function $g$ such that $\CP(\be):=\log p+g(\be)e^{-\gamma\be}$.
\end{proposition}
\begin{proof}
We have seen (Lemma \ref{lem-hb-const-aubry}) that $H_{\be}$ is constant on the sets $\S_{j}$. This shows that $V:=\lim_{\be\to+\8}\frac1\be\log H_{\be}$ is also constant of the $\S_{j}$. 
For simplicity we set $\uij:=u_{ij}u_{ij}u_{ij}\ldots$. This is a point in $\S_{j}$.  
Now we have 
\begin{equation}
\label{equ1-maxplusgene}
(e^{\CP(\be)}-p)H_{\be}(\uij)=\sum_{l\neq j}\sum_{i}e^{-\alpha_{(l-1)p+i}\theta\be}H_{\be}(u_{1l}*)\quad +e^{-\al}H_{\be}(u).
\end{equation}
Note that the results we get concerning the subaction $V$ are actually true for any calibrated subaction. We point out that we can first chose some subsequence of $\be$ such that $\disp\frac1\be\log(\CP(\be)-\log p)$ converges, and then take a new subsequence from the previous one to ensure that $\disp\frac1\be\log(H_{\be})$ also converges. 

We thus consider $V:=\lim_{\be\to+\8}\frac1\be\log H_{\be}$ and $-\gamma:=\lim_{\be\to+\8}\frac1\be\log (\CP(\be)-\log p)$. At that moment we do not claim that $\gamma$ has the exact value set in the Proposition. It is only an accumulation point for  $\disp \frac1\be\log (\CP(\be)-\log p)$. 
The convergence of $\disp \frac1\be\log (\CP(\be)-\log p)$ will follow from the uniqueness  of the value for $\gamma$.

Then \eqref{equ1-maxplusgene} and Lemma \ref{lem2-subac-1234} yield for every $j$,
\begin{equation}
\label{equ2-maxplusgene}
-\gamma-V(\uij)=\max\left(\max_{l\neq j}\left(V(u_{il})-\al_{(l-1)p+1}\theta\right),V(u)-\al\right).
\end{equation}
Consider the $k$ rows and $k+1$ columns matrix 
$$M_{1}:=\left(\begin{array}{cccccc}-\8 & -\al_{p+1}\theta & -\al_{2p+1}\theta & \ldots & -\al_{(N-1)p+1}\theta & -\al \\-\al_1\theta & -\8 & -\al_{2p+1}\theta &  &  & -\al \\\vdots &  &  &  &  & \vdots \\-\al_1\theta & \ldots &  &  & -\8 & -\al\end{array}\right).$$
Then, using the Max-Plus formalism we get 
\begin{equation}
\label{equ1-vecto-maxplus}
\left(\begin{array}{c}V(u_{11}^\8) -\gamma\\V(u_{12}^\8)-\gamma \\\vdots \\V(u_{1N}^\8)-\gamma\end{array}\right)=M_{1}\left(\begin{array}{c}V(u_{11}*) \\V(u_{12}*) \\\vdots \\V(u_{1N}*) \\V(u)\end{array}\right).
\end{equation}
Now, consider the  $k+1$ rows and $k$ columns matrix
$$M_{2}:=\left(\begin{array}{cccc}-\al_{1}\frac{\theta^2}{1-\theta} & -\al_{p+1}\frac{\theta}{1-\theta} &   & -\al_{(N-1)p+1}\frac{\theta}{1-\theta} \\
-\al_{1}\frac{\theta}{1-\theta} & -\al_{p+1}\frac{\theta^2}{1-\theta} &   &   \\
\vdots &    &   \ddots& \vdots   \\  &   &   & -\al_{(N-1)p+1}\frac{\theta^{2}}{1-\theta} \\
-\al_{1}\frac{\theta}{1-\theta} & -\al_{p+1}\frac{\theta}{1-\theta} &  \ldots & -\al_{(N-1)p+1}\frac{\theta}{1-\theta}\end{array}\right).$$
Then \eqref{equ2-suba-cal2} can be written as 
\begin{equation}
\label{equ2-vecto-maxplus}
\left(\begin{array}{c}V(u_{11}*) \\V(u_{12}*) \\\vdots \\V(u_{1N}*) \\V(u)\end{array}\right)= M_{2}\left(\begin{array}{c}V(u_{11}^\8) \\V(u_{12}^\8) \\\vdots \\V(u_{1N}^\8)\end{array}\right).
\end{equation}
\end{proof}
Combining \eqref{equ1-vecto-maxplus} and \eqref{equ2-vecto-maxplus}, we get that $-\gamma$ is an eigenvalue for the matrix $M_{1}M_{2}$ (for the Max-Plus algebra) and $\disp \left(\begin{array}{c}V(u_{11}^\8) \\V(u_{12}^\8) \\\vdots \\V(u_{1N}^\8)\end{array}\right) $ is an eigenvector.

Let us compute the matrix $M=M_{1}M_{2}$. Let us consider the row $l$ for $M_{1}$ and the column $j$ for $M_{2}$. Assume $j\neq l$. 

We have to compute the maximum between the sum of the $n^{th}$ term of the row and the $n^{th}$ term of the column. All the terms in the column are equal to $-\al_{(j-1)p+1}\frth$ except the $j^{th}$ which is $-\al_{(j-1)p+1}\frth.\theta$. This term is added to $-\al_{(j-1)p+1}\theta$ (the $j^{th}$ term of the column), and this addition gives $-\al_{(j-1)p+1}\frth$. Therefore, this term is the maximum (any other term is that one plus something negative).

Assume now that $j=l$. Then, the $j^{th}$ term of the column is added to $-\8$, hence disappears. Now, we just have to compute the maximum of all the terms respectively equal to a negative term minus $\al_{(j-1)p+1}\frth$. This means that we just have to take the maximal term in the row and to subtract $\al_{(j-1)p+1}\frth$.

Finally, the coefficient $m_{ij}$ of $M$ is 
$$m_{ij}=\left\{\begin{array}{l}
\max(-\al_{p+1}\theta,-\al)-\al_{1}\frth \text{ if } i=1=j,\\
\max(-\al_{1}\theta,-\al)-\al_{(j-1)p+1}\frth \text{ if } i=j\neq 1,\\
-\al_{(j-1)p+1}\frth \text{ if }i\neq j.
\end{array}\right.$$

To compute the eigenvalue for this matrix, we have to find the ``basic loop'' with biggest mean value. 

A basic loop is a word in $1,\ldots, k$ where no letter appear several times. Then we compute the mean value of the costs of the transition $i\to j$ given by the coefficient  $m_{ij}$ of the matrix for the letters of the basic loop. 

$\bullet$ Inequalities $\al_{1}<\al_{p+1}<\al_{jp+1}$ yields that any basic loop of length greater than 2 gives a lower contribution than the length 2-loop $1\to 2\to1$. This contribution is 
$$-\frac{\al_{1}+\al_{p+1}}2\frth.$$

$\bullet$ We claim that every basic loop of length 1 gives a smaller contribution than the first one. The claim is easy if $\alpha<\al_{1}\theta$. In that case we have 
$$\max(-\al_{p+1}\theta,-\al)-\al_{1}\frth \ge-\al-\al_{1}\frth>-\al-\al_{jp+1}\frth=\max(-\al_{1}\theta,-\al)-\al_{jp+1}\frth.$$
If $\al_{p+1}\theta\le \al$ the claim is also true:
$$-\al_{p+1}\theta-\al_{1}\frth> -\al_{1}\theta-\al_{jp+1}\frth$$
$$\Updownarrow \theta\al_{1}<  \al_{jp+1}+(\theta-1)\al_{p+1},$$
and this last inequality holds because 
$$\al_{jp+1}+(\theta-1)\al_{p+1}=\theta\al_{p+1}+(\al_{jp+1}-\al_{p+1})\ge\theta\al_{p+1}>\theta\al_{1}.$$
And finally, if $\al_{1}\theta\le \alpha\le \al_{p+1}\theta$, $\theta\al_{1}-\al$ is non-positive and $\theta\al_{jp}-\al$ is non-negative, and we let the reader check that this yields
$$-\al-\al_{1}\frth\ge-\al_{1}\theta-\al_{jp}\frth.$$
This shows 
$$-\gamma=\max(\max(-\al_{p+1}\theta,-\al)-\al_{1}\frth, -\frac{\al_{1}+\al_{p+1}}2\frth).$$
In particular, $\disp \frac1\be\log(\CP(\be)-\log p)$ has a unique accumulation point, hence converges. Then, there exists a sub-exponential function $g(\be)$ such that 
\begin{equation}
\label{equ-def-cpg}
\CP(\be)=\log p+g(\beta)e^{-\gamma.\be}.
\end{equation}
The pressure is  convex and analytic (the potential is Lipschitz continuous) and always bigger than $\log p$.  It is decreasing because its derivative is $\int A\,d\mu_{\beta}$ and $\mu_{\be}$ gives positive weight to any open set and $A$ is negative except on the empty interior sets $\S_{j}$. 
This proves that $g(\be)$ is positive.

\begin{remark}
\label{rem-gammaal1}
We emphasize $\gamma>\al_{1}\frth$.
\end{remark}

\subsection{Value for $\gamma$ in function of the parameters}
In this subsection we want to state an exact expression for $\gamma$ depending on the values for the parameters. We have seen 
$$\gamma=\min(\frac{\al_{1}+\al_{p+1}}2\frth,\al_{1}\frth+\al,\al_{1}\frth+\al_{p+1}\theta).$$
Now, $\disp \al_{1}\frth+\al_{p+1}\theta<\frac{\al_{1}+\al_{p+1}}2\frth$ if and only if $\al_{p+1}>\disp\frac{\al_{1}}{2\theta-1}$ (which is possible only for $\theta>\frac12$). 
Obviously, $\disp \al_{1}\frth+\al<\al_{1}\frth+\al_{p+1}\theta$ means $\al<\al_{p+1}\theta$. 

\begin{figure}[htbp]
\begin{center}
\includegraphics[scale=0.5]{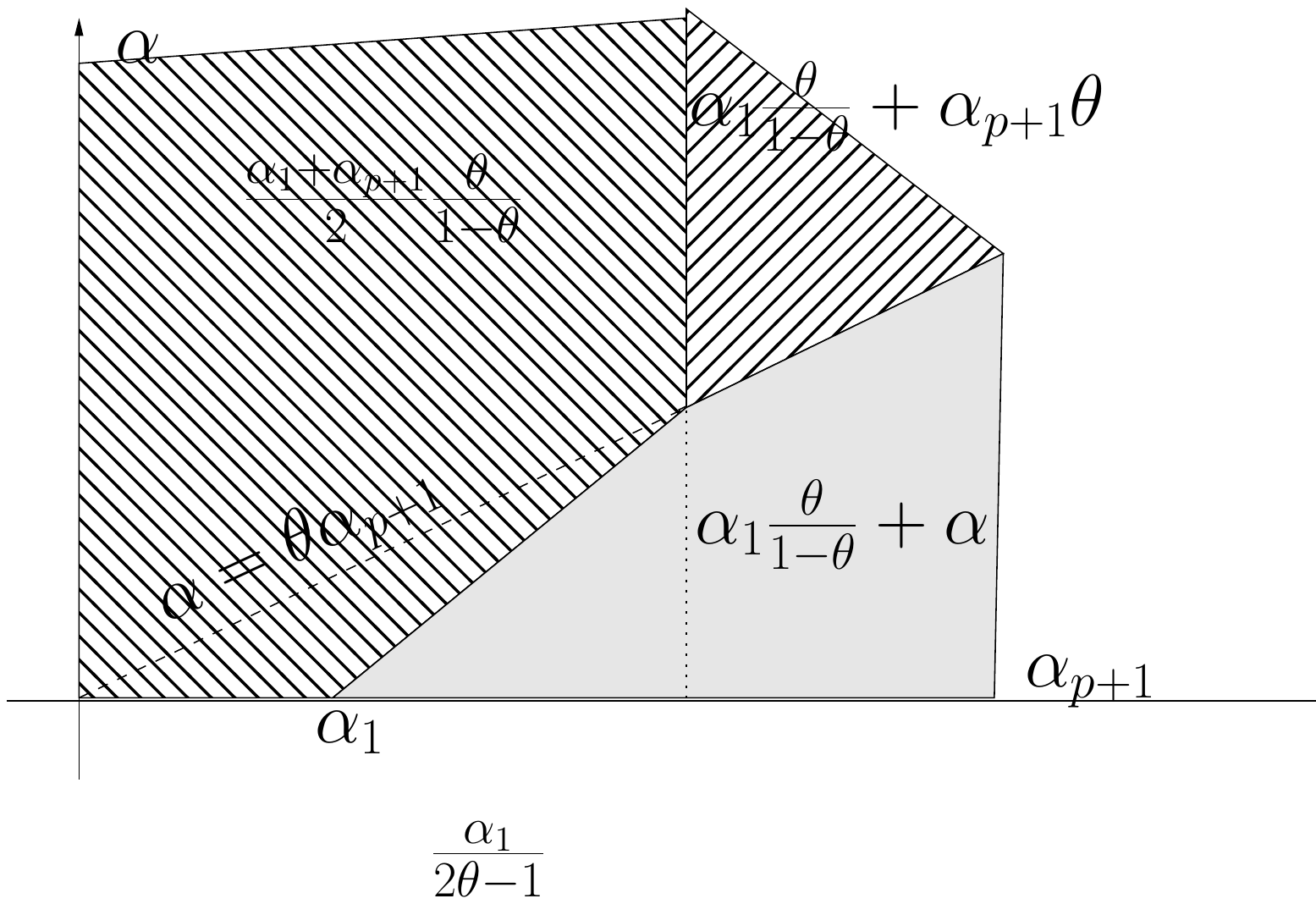}
\caption{Values for $\gamma$}
\label{fig-gamma}
\end{center}
\end{figure}

Finally,  $\disp \al_{1}\frth+\al<\frac{\al_{1}+\al_{p+1}}2\frth$ means $\al<\disp\frac{\al_{p+1}-\al_{1}}2\frth$. Note that for $\al_{p+1}=\disp\frac{\al_{1}}{2\theta-1}$, 
$$\frac{\al_{p+1}-\al_{1}}2\frth=\al_{1}\frac\theta{2\theta-1}=\theta\al_{p+1}.$$

\section{Auxiliary function F}\label{sec-F}

\begin{lemma}
\label{lem-prodxi1xi2}
Let $0<\xi_{1}<\xi_{2}\le \ldots \le\xi_{p}$ be $p$ positive real numbers ($p\ge 2$). Let us set
\begin{itemize}
\item $\eta_{i}:=\xi_{i}-\xi_{1}$, for $i\ge2$,
\item  $\disp r:=\frac{-\log p}{\log\theta}(>0)$,
\item $\disp I(\eta_{2},\eta_{3},\ldots,\eta_{p}):=\int_{0}^{1}\log\left(1+\frac{\sum_{i=2}^{p}e^{-\eta_{i}x}-1}p\right)\frac{dx}x+\int_{1}^{+\8}\frac{\sum_{i=2}^{p}\eta_{i}e^{-\eta_{i}x}}{1+\sum_{i=2}^{p}e^{-\eta_{i}x}}\log x\,dx$. 
\end{itemize}
Then, if $n$ goes first to $+\8$ and then $\beta$ goes to $+\8$,

\begin{eqnarray*}
\prod_{j=1}^{n}\left(e^{-\xi_{1}\theta^{j}\be}+e^{-\xi_{2}\theta^{j}\be}+\ldots+e^{-\xi_{p}\theta^{j}\be}\right)=\frac{p^{n}}{\be^{r}}e^{
-\xi_{1}\frac\theta{1-\theta}\beta(1-\theta^{n})-\frac{I(\eta_{2},\ldots,\eta_{p})}{\log \theta}+O(\beta\theta^{n})+o_{\8}(\beta)},
\end{eqnarray*}
where $O(\be\theta^{n})$ is bounded in absolute value by a term of the form $C\sum_{i=2}^{p}\eta_{i}\be\theta^{n}$ for some universal constant $C$ and $o_{\infty}(\be)$ is bounded in absolute value by a term of the form $C'\sum_{i}\eta_{i}e^{-\frac{\eta_{2}\be\theta}2}$ for some universal constant $C'$. 
\end{lemma}

\begin{proof}
First we write 
\begin{equation}\label{equ-x1factor}
\prod_{j=1}^{n}\left(e^{-\xi_{1}\theta^{j}\be}+e^{-\xi_{2}\theta^{j}\be}+\ldots+e^{-\xi_{p}\theta^{j}\be}\right)=e^{-\xi_{1}\be\theta\frac{1-\theta^n}{1-\theta}}\prod_{j=1}^{n}\left(1+e^{-\eta_{2}\theta^{j}\be}+\ldots+e^{-\eta_{p}\theta^{j}\be}\right),
\end{equation}
and
\begin{eqnarray*}
\log\prod_{j=1}^{n}\left(1+e^{-\eta_{2}\theta^{j}\be}+\ldots+e^{-\eta_{p}\theta^{j}\be}\right)&=& \sum_{j=1}^{n}\log\left(1+e^{-\eta_{2}\theta^{j}\be}+\ldots+e^{-\eta_{p}\theta^{j}\be}\right) \\
&=& n\log p+\sum_{j=1}^{n}\log\left(1+\frac{\sum_{i=2}^{p}e^{-\eta_{i}\theta^{j}\be}-1}p\right).
\end{eqnarray*}
Note that $\eta_{i}\theta^{j}$ decreases in $j$. Thus we can compare this later sum with an integral
\begin{eqnarray*}
\int_{0}^{n}\log\left(1+\frac{\sum_{i=2}^{p}e^{-\eta_{i}\theta^{x}\be}-1}p\right)\,dx&\le &
\sum_{j=1}^{n}\log\left(1+\frac{\sum_{i=2}^{p}e^{-\eta_{i}\theta^{j}\be}-1}p\right)\\
&&\hskip -2cm\le  \int_{1}^{n+1}\log\left(1+\frac{\sum_{i=2}^{p}e^{-\eta_{i}\theta^{x}\be}-1}p\right)\,dx.
\end{eqnarray*}
Let $I_{n}$ and $J_{n}$ respectively denote the integral from left hand side and the right hand side. First, we focus the study on $I_{n}$.

In order to study $I_{n}$, let us set $u=\be\theta^{x}$. Then we have 
$$I_{n}=\frac1{\log\theta}\int_{\be}^{\be\theta^{n}}\log\left(1+\frac{\sum_{i=2}^{p}e^{-\eta_{i}u}-1}p\right)\,\frac{du}u,$$
and we split this last integral in two pieces $\disp\int_{1}^{\be\theta^{n}}$ and $\disp\int_{\be}^{1}$.

We remind that $n$ is supposed to go first to $+\8$ and then $\be$ goes to $+\8$. Hence, $\be\theta^{n}$ is close to 0. For $u$ close to 0, $\disp \log\left(1+\frac{\sum_{i=2}^{p}e^{-\eta_{i}u}-1}p\right)$ is non-positive and bigger than a term of the form $\disp -C\sum_{i=2}^{p}\eta_{i}u$ for some universal constant $C$. This shows that the integral 
\begin{eqnarray*}
\int_{1}^{0}\log\left(1+\frac{\sum_{i=2}^{p}e^{-\eta_{i}u}-1}p\right)\,\frac{du}u,
\end{eqnarray*}
converges (the function has a limit as $u$ goes to $0$) and 
\begin{equation}
\label{equ1-inteIn-fine}
\int_{1}^{\be\theta^{n}}\log\left(1+\frac{\sum_{i=2}^{p}e^{-\eta_{i}u}-1}p)\right)\,\frac{du}u=\int_{1}^{0}\log\left(1+\frac{\sum_{i=2}^{p}e^{-\eta_{i}u}-1}p)\right)\,\frac{du}u +O(\be\theta^{n}),
\end{equation}
where $\disp \left|O(\be\theta^{n})\right|\le C\sum_{i=2}^{p}\be\theta^{n}$.

For the other part we get:
\begin{eqnarray*}
\int_{\be}^{1}\log\left(1+\frac{\sum_{i=2}^{p}e^{-\eta_{i}u}-1}p)\right)\,\frac{du}u&=& 
\left[\log\left(1+\frac{\sum_{i=2}^{p}\eta_{i}e^{-\eta_{i}u}-1}p)\right)\log u\right]_{\be}^{1}+\int_{\be}^{1}\frac{\sum_{i=2}^{p}e^{-\eta_{i}u}}{1+\sum_{i=2}^{p}e^{-\eta_{i}u}}\log u\,du\\
&=& -\log\left(1+\frac{\sum_{i=2}^{p}\eta_{i}e^{-\eta_{i}\be}-1}p)\right)\log \be+\int_{\be}^{1}\frac{\sum_{i=2}^{p}e^{-\eta_{i}u}}{1+\sum_{i=2}^{p}e^{-\eta_{i}u}}\log u\,du\\
&=& \log p\log\be-\log\left(1+\sum_{i=2}^{p}\eta_{i}e^{-\eta_{i}\be})\right)\log \be+\\
&&\int_{+\8}^{1}\frac{\sum_{i=2}^{p}e^{-\eta_{i}u}}{1+\sum_{i=2}^{p}e^{-\eta_{i}u}}\log u\,du+
\int_{\be}^{+\8}\frac{\sum_{i=2}^{p}e^{-\eta_{i}u}}{1+\sum_{i=2}^{p}e^{-\eta_{i}u}}\log u\,du.
\end{eqnarray*}
Now, both terms $\disp\left| \log\left(1+\sum_{i=2}^{p}\eta_{i}e^{-\eta_{i}\be})\right)\log \be\right|$ and $\disp\left|\int_{\be}^{+\8}\frac{\sum_{i=2}^{p}e^{-\eta_{i}u}}{1+\sum_{i=2}^{p}e^{-\eta_{i}u}}\log u\,du\right|$ are bounded from above by some $\disp C'\sum_{i}\eta_{i}e^{-\frac{\eta_{2}}2\be}$.

\medskip
The computation for $J_{n}$ is similar except that borders have to be exchanged. Namely $\be\theta^{n}$  in $\disp\int_{1}^{\be\theta^{n}}$ has to be replaced by $\be\theta^{n+1}$ which improves the estimate, and $\be$ in $\disp \int_{\be}^{1}$ has to be replaced  by $\int_{\be\theta}^{1}$. This produce an upper bound of the form $O(e^{-\frac{\eta_{2}}2\theta\be})$ instead of $O(e^{-\frac{\eta_{2}}2\be})$.

This concludes the proof of the lemma.
\end{proof}

\begin{definition}
\label{def-fun-F}
We define the auxiliary function $\disp F(Z,z_{1},\ldots, z_{p})$ by 
$$F(Z,z_{1},\ldots, z_{p}):=\sum_{n=1}^{+\8}\left(e^{-nZ}\prod_{j=1}^{n}(e^{-z_{1}\theta^{j}}+\ldots+e^{-z_{p}\theta^{j}})\right).$$
For an integer $K$, $F_{K}()$ denotes the truncated sum to $K$:
$$F_{K}(Z,z_{1},\ldots, z_{p}):=\sum_{n=1}^{K}\left(e^{-nZ}\prod_{j=1}^{n}(e^{-z_{1}\theta^{j}}+\ldots+e^{-z_{p}\theta^{j}})\right).$$
\end{definition}

\begin{proposition}
\label{prop-equil-F} 
Let $0<\xi_{1}<\xi_{2}\le \ldots \le\xi_{p}$ be $p$ positive real numbers ($p\ge 2$). We re-employ notations from Lemma \ref{lem-prodxi1xi2}.

Then, as $\be$ goes to $+\8$
$$\text{ if }\gamma>\xi_{1}\frth, \text{ then }F(\CP(\be),\xi_{1}\be,\ldots, \xi_{p}\be)=\frac1{\beta^{r}g(\be)}e^{(\gamma-\xi_{1}\frth)\be-\frac{I(\eta_{2},\ldots,\eta_{p})}{\log\theta}+o_{\8}(\be)},$$

$$\text{ if }\gamma<\xi_{1}\frth, \text{ then }F(\CP(\be),\xi_{1}\be,\ldots, \xi_{p}\be)=O(e^{-\xi_{1}\theta\be}\vee \frac1{\beta^{r}g(\be)}e^{(\gamma-\xi_{1}\frth)\be}),$$
where $o_{\8}(\be)$ goes to 0 as $\be$ goes to $+\8$.
\end{proposition}
\begin{proof}
Let $\eps_{0}$ be a positive real number such that $\log\eps_{0}<-2$. 
We set 
\begin{equation}
\label{equ-def-nbeta}
n(\beta):=\frac{\log\eps_{0}}{\log\theta}\log\be.
\end{equation}

 Note that $n(\be)$ goes to $+\8$ as $\be$ goes to $+\8$. Furthermore, $\theta^{n(\beta)}<\frac1{\be^{2}}$ and $n\be\theta^{n}$ goes to $0$ if $n\ge n(\be)$ and $\beta$ goes to $+\8$.

The proof has three steps.
The function $F$ is defined as a sum over $n$, for $n\ge 1$. In the first part, we give bounds for a fixed $\beta$,  and for the sum for $n\ge n(\be)$. This quantity is a trivial bound from below for the global sum.

In the second step we prove that the sum for $n\le n(\be)-1$ goes to $0$ as $\be$ goes to $+\8$. This allows to conclude the proof for the case $\gamma>\xi_{1}\frth$.

In the last step, we use the computations of the second step to conclude the proof for the case $\gamma<\xi_{1}\frth$. 

{\bf First step}. 
Remember that $\CP(\be)=\log p+g(\be)e^{-\gamma.\be}$, where $g(\be)$ is a positive and sub-exponential function in $\be$. 
Then Lemma \ref{lem-prodxi1xi2} yields for $n\ge n(\be)$,

\begin{eqnarray*}
\frac1{\be^{r}}\exp\left(-ng(\be)e^{-\gamma.\be}-\xi_{1}\frac\theta{1-\theta}\beta(1-\theta^{n})-\frac{I(\eta_{2},\ldots,\eta_{p})}{\log \theta} +O(\beta\theta^{n})+o_{\8}(\beta)\right)\\
= e^{-n\CP(\be)}\prod_{j=1}^{n}\left(e^{-\xi_{1}\theta^{j}\be}+e^{-\xi_{2}\theta^{j}\be}+\ldots+e^{-\xi_{p}\theta^{j}\be}\right).
\end{eqnarray*}
As we only consider $n\ge n(\be)$ and $\be$ goes to $+\8$, we can replace $O(\be\theta^{n})$ and $\theta^{n}$ by $o_{\8}(\be)$. 
Doing the sum over $n$, only the terms $e^{-ng(\beta)e^{-\gamma.\be}}$ have to be summed. We thus get 

\begin{eqnarray}
\frac1{\be^{r}}\exp\left(-\xi_{1}\frac\theta{1-\theta}\beta-\frac{I(\eta_{2},\ldots,\eta_{p})}{\log \theta}+o_{\8}(\beta)\right)\sum_{n\ge n(\be)}e^{-ng(\be)e^{-\gamma.\be}}\nonumber\\
=\sum_{n\ge n(\be)} e^{-n\CP(\be)}\prod_{j=1}^{n}\left(e^{-\xi_{1}\theta^{j}\be}+e^{-\xi_{2}\theta^{j}\be}+\ldots+e^{-\xi_{p}\theta^{j}\be}\right).\label{equ-estimF-1}
\end{eqnarray}
Now, $\disp \sum_{n\ge n(\be)}e^{-n(\be)g(\be)e^{-\gamma.\be}}=\frac{e^{-n(\be)g(\be)e^{-\gamma.\be}}}{1-e^{-g(\be)e^{-\gamma.\be}}}$. Both $g(\be)$ and $n(\be)$ are sub-exponential in $\be$, hence the numerator goes to 1 as $\be$ goes to $+\8$. 
The denominator behaves like $g(\be)e^{-\gamma.\be}$. 
Then, \eqref{equ-estimF-1} yields
\begin{eqnarray}
\frac{1}{\be^{r}g(\be)}e^{(\gamma-\xi_{1}\frac\theta{1-\theta})\beta-\frac{I(\eta_{2},\ldots,\eta_{p})}{\log \theta}+o_{\8}(\beta)}
=\sum_{n\ge n(\be)} e^{-n\CP(\be)}\prod_{j=1}^{n}\left(e^{-\xi_{1}\theta^{j}\be}+e^{-\xi_{2}\theta^{j}\be}+\ldots+e^{-\xi_{p}\theta^{j}\be}\right)\nonumber\\
\label{equ-estimF-2}.
\end{eqnarray}

{\bf Second step.}
All the $\xi_{j}$ are bigger than $\xi_{1}$. We thus trivially get 
$$e^{-n\CP(\be)}\prod_{j=1}^{n}\left(e^{-\xi_{1}\theta^{j}\be}+\ldots +e^{-\xi_{p}\theta^{j}\be}\right)\le e^{-\xi_{1}\frth\beta}e^{-ng(\be)e^{-\gamma.\be}+\xi_{1}\frth\beta\theta^{n}}.$$
For the rest of the proof, we set $D:=g(\be)e^{-\gamma.\be}$ and $E:=\xi_{1}\frth\beta$. 
The sequence $-nD+E\theta^{n}$ decreases in $n$, and we can (again) compare the sum with the integral. 

We get 
$$F_{n(\be)-1}(\CP(\be),\xi_{1}\be,\ldots,\xi_{p}\be)\le e^{-E}\sum_{n=1}^{n(\be)-1}e^{-nD+E\theta^{n}}\le e^{-E}\int_{0}^{n(\be)}e^{-xD+E\theta^{x}}\,dx.$$
Now we get

\begin{eqnarray*}
\int_{0}^{n(\be)}e^{-xD+E\theta^{x}}\,dx&=& \frac1{\log\theta}\int_{1}^{\theta^{n(\beta)}}\frac{e^{-\frac{D\log u}{\log\theta}+Eu}}u\,du\\
&=& \frac1{\log\theta}\int_{1}^{\theta^{n(\beta)}}u^{-1-\frac{D}{\log\theta}}e^{Eu}\,du\\
&=&  \frac1{\log\theta}\sum_{k=0}^{+\8}\frac{E^{k}}{k!}\int_{1}^{\theta^{n(\beta)}}u^{k-1-\frac{D}{\log\theta}}\,du\\
&=& \frac1{\log\theta}\sum_{k=0}^{+\8}\frac{E^{k}}{k!}\frac1{k-\frac{D}{\log\theta}}\left[\theta^{n(\be)(k-\frac{D}{\log\theta})}-1\right].\\
&=& \frac1D\left(1-\theta^{-n(\be)\frac{D}{\log\theta}}\right)+\frac1{\log\theta}\sum_{k=1}^{+\8}\frac{E^{k}}{k!}\frac1{k-\frac{D}{\log\theta}}\left[\theta^{n(\be)(k-\frac{D}{\log\theta})}-1\right].
\end{eqnarray*}

As we shall consider $\be$ close to $+\8$ we can assume that $\be$ is big enough such that $\theta^{n(\be)}<\frac12$. Then we have 
$$
 \frac1{\log\theta}\int_{1}^{\theta^{n(\beta)}}\frac{e^{-\frac{D\log u}{\log\theta}+Eu}}u\,du= \frac1{\log\theta}\int_{1}^{\frac12}\frac{e^{-\frac{D\log u}{\log\theta}+Eu}}u\,du+ \frac1{\log\theta}\int_{\frac12}^{\theta^{n(\beta)}}\frac{e^{-\frac{D\log u}{\log\theta}+Eu}}u\,du.
$$
Let us first study the last integral. 
\begin{eqnarray*}
 0\le\frac1{\log\theta}\int_{\frac12}^{\theta^{n(\beta)}}\frac{e^{-\frac{D\log u}{\log\theta}+Eu}}u\,du&\le & \frac1{\log\theta}\int_{\frac12}^{\theta^{n(\beta)}}u^{-1-\frac{D}{\log\theta}}e^{\frac{E}2}\,du\\
 &\le & e^{\frac{E}2}\frac{1}{D}\left[\left(\frac12\right)^{-\frac{D}{\log\theta}}-\theta^{-n(\be)\frac{D}{\log\theta}}\right].
\end{eqnarray*}
Remember that $D=g(\be)e^{-\gamma\be}$ and $n(\be)=\frac{\log\eps_{0}}{\log\theta}\log\be$. Hence, $D$ and $n(\be)D$ go to 0 as $\be$ goes to $+\8$. 
Then 
\begin{eqnarray*}
\frac1D\left[\left(\frac12\right)^{-\frac{D}{\log\theta}}-\theta^{-n(\be)\frac{D}{\log\theta}}\right]&=&
\frac1D\left[\left(\frac12\right)^{-\frac{D}{\log\theta}}-1\right]+\frac1D\left[1-\theta^{-n(\be)\frac{D}{\log\theta}}\right]\\
&=& \frac{\log 2}{\log\theta}+n(\be)+o(D).
\end{eqnarray*}
Let us now study $\disp \frac1{\log\theta}\int_{1}^{\frac12}\frac{e^{-\frac{D\log u}{\log\theta}+Eu}}u\,du$. 
\begin{eqnarray*}
0\le \frac1{\log\theta}\int_{1}^{\frac12}\frac{e^{-\frac{D\log u}{\log\theta}+Eu}}u\,du&=&
\frac1{\log\theta}\int_{1}^{\frac12}u^{-1-\frac{D}{\log\theta}}e^{Eu}\,du\\
&\le & \frac{2^{1+\frac{D}{\log\theta}}}{E|\log\theta|}\left[e^{E}-e^{\frac{E}2}\right].
\end{eqnarray*}

Remember that  $E=\xi_{1}\frth\be$. Therefore we finally get 
\begin{equation}
\label{equ-majofine-Fnbe}
F_{n(\be)-1}(\CP(\be),\xi_{1}\be,\ldots,\xi_{p}\be)\le n(\be)e^{-\frac{\xi_{1}}2\frac\theta{1-\beta}\be}+\frac\kappa{\xi_{1}\beta}, 
\end{equation}
for some universal constant\footnote{We emphasize that $\kappa$ can be assume to be smaller than 4 if $\be$ is chosen sufficiently big.} $\kappa>0$. 

The term in the right hand side in \eqref{equ-majofine-Fnbe} goes to 0 as $\be$ goes to $+\8$. In particular, \eqref{equ-estimF-2} and \eqref{equ-majofine-Fnbe} show that the result holds  if $\gamma>\xi_{1}\frth$ because the sum for $n\le n(\be)$ is negligible with respect to the sum for $n>n(\be)$.

{\bf Third step.}  We assume $\gamma<\xi_{1}\frth$. Let $K$ be sufficiently big such that $\xi_{1}\frth\theta^{K}<\gamma$. Then, note that we get 
\begin{eqnarray}
F(\CP(\be),\xi_{1}\be,\ldots,\xi_{p}\be)&=&F_{K}(\CP(\be),\xi_{1}\be,\ldots,\xi_{p}\be)+\nonumber\\
&&\hskip -2cm e^{-K\CP(\be)}\prod_{j=1}^{K}\left(e^{-\xi_{1}\theta^{j}\be}+\ldots+e^{-\xi_{p}\theta^{j}\be}\right)F(\CP(\be),\xi_{1}\theta^{K}\be,\ldots,\xi_{p}\theta^{K}\be).\label{equ1-majo-debut-F}
\end{eqnarray}
Again we have 
$$ \prod_{j=1}^{K}\left(e^{-\xi_{1}\theta^{j}\be}+\ldots+e^{-\xi_{p}\theta^{j}\be}\right)=e^{-\xi_{1}\frth\be+\xi_{1}\theta^{K}\frth\be}\prod_{j=1}^{K}\left(1+e^{-\eta_{2}\theta^{j}\be}+\ldots+e^{-\eta_{p}\theta^{j}\be}\right).$$
This last term $\disp \prod_{j=1}^{K}\left(1+e^{-\eta_{2}\theta^{j}\be}+\ldots+e^{-\eta_{p}\theta^{j}\be}\right)$ goes to $1$ as $\beta$ goes to $+\8$. The term $e^{-K\CP(\be)}$ goes to the constant $p^{K}$, and $F(\CP(\be),\xi_{1}\theta^{K}\be,\ldots,\xi_{p}\theta^{K}\be)$ behaves (at the exponential scale) like $O(e^{(\gamma-\xi_{1}\theta^{K}\frth)\be})$. This proves that the second term in the right hand side of \eqref{equ1-majo-debut-F} behaves like  $O(e^{(\gamma-\xi_{1}\frth)\be})$.

Now, the finite sum $F_{K}(\CP(\be),\xi_{1}\be,\ldots,\xi_{p}\be)=e^{-\CP(\be)-\xi_{1}\theta\be}+$ terms which are exponentially small with respect to $e^{-\xi_{1}\theta\be}$ if $\be$ goes to $+\8$. Hence, this finite sum behaves like the biggest term, namely like $O(e^{-\xi_{1}\theta\be})$. This concludes the proof of the lemma. 
\end{proof}

\section{Proof of Theorem 1: Convergence for the eigenmeasure $\nu_{\be}$}\label{sec-proof-th1}
\subsection{Selection for $\nu_{\be}$}
 We set $O_{j}:=\sqcup_{i}[u_{ij}]$. This is the set of points whose first digit is one of the $u_{ij}$'s of the alphabet $\CA_{j}$. 
 \begin{proposition}
\label{prop-value-nube}
For every $j$, 
$$\nu_{\be}(O_{j})=\frac{ F(\CP(\be),\al_{(j-1)p+1}\be,\ldots,\al_{jp}\be)e^{-\CP(\be)}}{1+ F(\CP(\be),\al_{(j-1)p+1}\be,\ldots,\al_{jp}\be)e^{-\CP(\be)}}.$$
\end{proposition}
 \begin{proof}
Let $m$ be an admissible word for $\S_{j}$ with length $n$. Let $k\notin \CA_{j}$. Then, 
\begin{equation}
\label{equ-nube-mot}
\nu_{\be}([mk])=e^{-n\CP(\be)-\sum_{j=0}^{n-1}\alpha_{m_{j}}\theta^{n-j}}\nu_{\be}([k]).
\end{equation}
This yields 
\begin{eqnarray*}
\nu_{\be}(O_{j})&=&\sum_{k\notin\CA_{j}} \sum_{n=1}^{+\8}\sum_{m,|m|=n}e^{-n\CP(\be)-\sum_{j=0}^{n-1}\alpha_{m_{j}}\theta^{n-j}}\nu_{\be}([k])\\
&=& \sum_{n=1}^{+\8}\left(e^{-n\CP(\be)}\prod_{j=1}^{n}\left(e^{-\alpha_{1}\be\theta^{j}}+\ldots +e^{-\alpha_{p}\be\theta^{j}}\right)\right)e^{-\CP(\be)}(1-\nu_{\be}(O_{j}))\\
&=& F(\CP(\be),\al_{1}\be,\ldots,\al_{p}\be)e^{-\CP(\be)}(1-\nu_{\be}(\sqcup_{i=1}^{p}[i]))
\end{eqnarray*}
Therefore we get 
$$\nu_{\be}(O_{j})=\frac{ F(\CP(\be),\al_{1}\be,\ldots,\al_{p}\be)e^{-\CP(\be)}}{1+ F(\CP(\be),\al_{1}\be,\ldots,\al_{p}\be)e^{-\CP(\be)}}.$$

\end{proof}

We also have an explicit value for $\nu_{\be}([u])$: we write 
$$[u]=[uu]\bigsqcup_{i=1}^{kp}[u\,i],$$
and use conformity to get 
\begin{equation}
\label{equ-value-nube-u}
\disp \nu_{\be}([u])=e^{-\CP(\be)-\al\be}.
\end{equation} In particular this quantity goes to 0 as $\be$ goes to $+\8$. This shows that only the $N$ subshifts $\S_{j}$ can have positive measure as $\beta$ goes to $+\8$.

Now, Propositions \ref{prop2-maxplus} and \ref{prop-equil-F} show that 
$F(\CP(\be),\al_{1}\be,\ldots,\al_{p}\be)$  behaves like $e^{(\gamma-\alpha_{1}\frth)\be}$ as $\be$ goes to $+\8$, and this quantity diverges to $+\8$. 
It also shows that for every $j>1$, $F(\CP(\be),\al_{(j-1)p+1},\ldots, \al_{jp})$ goes exponentially fast to $0$.

This shows that $\nu_{\be}(O_{1})$ goes to 1. The speed of the convergence is given by $\disp \frac1{1+F(\CP(\be),\al_{1}\be,\ldots, \al_{p}\be)}$ which goes exponentially fast to $0$.

As a by-product, we also immediately get from \eqref{equ-nube-mot} that for any $\S_{1}$-admissible word $m$ with $|m|=n$ and for any $k\notin\CA_{1}$, 
$$\nu_{\be}([mk])=e^{-n\CP(\be)-\sum_{j=0}^{n-1}\alpha_{m_{j}}\theta^{n-j}}\nu_{\be}([k])\rightarrow_{\be\to+\8}0,$$
as $\nu_{\be}([k])$ goes to $0$ if $\be$ goes to $+\8$. Hence we get 
\begin{proposition}
\label{prop-limitnube-S1}
Any accumulation point of $\nu_{\be}$ is a probability measure with support in $\S_{1}$. 
\end{proposition}

\subsection{Convergence for $\nu_{\be}$}
The measure of maximal entropy $\mu_{top,1}$ is the product of the eigenfunction and the eigenmeasure both associated to the transfer operator for $(\S_{1},\s)$ and the constant potential zero. Hence, the eigenmeasure $\nu_{top,1}$ is characterized by the fact that every $\S_{1}$-admissible words with a fixed length $n$ have the same measure $\disp\frac1{p^{n}}$. 

We have already seen above that any  accumulation point for $\nu_{\be}$ is a measure, say $\nu_{\8}$, such that $\nu_{\8}(O_{1})=0$. 
Our strategy to prove that $\nu_{\be}$ converges to $\nu_{top,1}$ is now to prove that for any two $\S_{1}$-admissible  words $m$ and $\wh m$ with the same length, 
$\disp\frac{\nu_{\be}([m])}{\nu_{\be}([\wh m])}$ goes to 1 as $\be$ goes to $+\8$. 

\bigskip
Let us thus consider two $\S_{1}$-admissible words $m$ and $\wh m$ with length $n$. In the following, $m'$ is a word (possibly the empty word) admissible for $\S_{1}$. We get 
\begin{eqnarray}
\nu_{\be}([m])&=& \sum_{m'}  \nu_{\be}(mm'*)\nonumber\\
&=& \sum_{k=0}^{+\8}e^{-n\CP(\be)}e^{S(m)\be \theta^{k}}e^{-k\CP(\be)}\prod_{j=1}^{k}(e^{-\al_{1}\theta^{j}\be}+\ldots+e^{-\al_{p}\theta^{j}\be})(1-\nu_{\be}(O_{1}))\nonumber\\
&& \hskip 1.5cm \text{where }S(m):=-\sum_{l=0}^{n-1}\al_{m_{l}}\theta^{n-l}\nonumber\\
&=& e^{-n\CP(\be)}\left(e^{S(m)\be }+ \sum_{k=1}^{+\8}e^{S(m)\be \theta^{k}}e^{-k\CP(\be)}\prod_{j=1}^{k}(e^{-\al_{1}\theta^{j}\be}+\ldots+e^{-\al_{p}\theta^{j}\be})\right)(1-\nu_{\be}(O_{1})).\nonumber\\
&&\label{equ1-nube-topo}
\end{eqnarray}
The series in the right hand side of this last equality is almost the same than the one defining $F(\CP(\be),\al_{1}\be,\ldots, \al_{p}\be)$ up to the extra term $e^{S(m)\be\theta^{k}}$. 

Replacing $m$ by $\wh m$ we get a similar formula for $\nu_{\be}([\wh m])$.

The quantities $S(m)$  and $S(\wh m)$ are negative, hence $e^{S(\wt m)\be\theta^{k}}$ is lower than 1 for $\wt m=m,\wh m$. Now, remember the definition of $n(\be)=\disp\frac{\log\eps_{0}}{\log\theta}\log\be$ (see p. \pageref{equ-def-nbeta}). Step 2 of the proof of Proposition  \ref{prop-equil-F} shows that $F_{n(\be)}(\CP(\be),\al_{1}\be,\ldots,\al_{p}\be)$ goes to $0$ as $\be$ goes to $+\8$, whereas  Step 1 of the proof of Proposition  \ref{prop-equil-F} shows that the tail $\disp \sum_{k\ge n(\be)+1}e^{-k\CP(\be)}\prod_{j=1}^{k}(e^{-\al_{1}\theta^{j}\be}+\ldots+e^{-\al_{p}\theta^{j}\be})$ diverges (exponentially fast) to $+\8$. 
Note that for $k\ge n(\be)$, $\disp \be\theta^{k}\le \frac1{\be}$. Therefore we get 
$$\frac{\nu_{\be}([m])}{\nu_{\be}([\wh m])}\le \frac{1+F_{n(\be)}(\CP(\be),\al_{1}\be,\ldots, \al_{p}\be)+\disp\sum_{k>n(\be)}e^{-k\CP(\be)}\prod_{j=1}^{k}(e^{-\al_{1}\theta^{j}\be}+\ldots+e^{-\al_{p}\theta^{j}\be})}{e^{\frac{S(\wh m)}{\be}}\disp\sum_{k>n(\be)}e^{-k\CP(\be)}\prod_{j=1}^{k}(e^{-\al_{1}\theta^{j}\be}+\ldots+e^{-\al_{p}\theta^{j}\be})}.$$
Doing $\be$ goes to $+\8$ in this last inequality we get $\disp\limsup_{\be\to+\8}\frac{\nu_{\be}([m])}{\nu_{\be}([\wh m])}\le 1$. Exchanging $m$ and $\wh m$ we also get 
$\disp\limsup_{\be\to+\8}\frac{\nu_{\be}([\wh m])}{\nu_{\be}([m])}\le 1$, which means 
$$\lim_{\be\to+\8}\frac{\nu_{\be}([m])}{\nu_{\be}([\wh m])}=1.$$
In other words, any accumulation point for $\nu_{\be}$ is a probability measure with support in $\S_{1}$ which gives the same weight to all the cylinders of same length. There exists only one such measure, it is $\nu_{top,1}$. This finishes the proof of Theorem 1.

\section{Proof of Theorem 2: convergence and selection for $\mu_{\beta}$}\label{sec-proof-th2}

\subsection{An expression for $\mu_{\be}(O_{j})$}

We recall that $n(\be)$ was defined (see Equation \ref{equ-def-nbeta}) by 
$$n(\be):=\frac{\log\eps_{0}}{\log\theta}\log\be.$$
Its main properties are that $n(\be)$ goes to $+\8$ as $\be$ goes to $+\8$ and $n(\be).\be.\theta^{\be}$  goes to $0$.

We recall that for every $j$ and for every $\S_{j}$-admissible word $m$ with finite length we have 
\begin{equation}
\label{equ-H-m1}
H_{\be}(m)=e^{-\CP(\be)}\sum_{i=1}^{p}e^{\be.A(u_{ij}m)}H_{\be}(u_{ij}m)+e^{-\CP(\be)}\tau_{j}.
\end{equation}

The main result in this subsection is the following proposition, which gives an expression for $\mu_{\be}(O_{j})$. We employ notations from Lemma \ref{lem-prodxi1xi2} and Proposition \ref{prop-equil-F}; we remind $\disp r=-\frac{\log p}{\log\theta}$ and that $I(\eta_{2},\ldots, \eta_{p})$ was defined there. 
\begin{proposition}
\label{prop-value-mubeoj} For every $j$, let us set $\eta_{ij}:=\al_{(j-1)p+i}-\al_{(j-1)p+1}$.
Then
$$\mu_{\be}(O_{j})=e^{-\CP(\be)}.\tau_{j}(1-\nu_{\be}(O_{j}))\left[\frac{1}{\al_{(j-1)p+1}}O\left(\frac{n(\be)}{\beta}\right)+\frac{n(\be)(1+o_{\8}(\be))}{\be^{r}\left(g(\be)\right)^{2}}e^{(2\gamma-\al_{(j-1)p+1}\frth)\be}e^{-\frac{I(\eta_{2j},\ldots,\eta_{pj})}{\log\theta}}\right],$$
where $o_{\8}(\be)$  goes to 0 as $\be$ goes to $+\8$. 
\end{proposition}

As $n(\be)$ is proportional to $\log\be$, the term $\disp O\left(\frac{n(\be)}{\beta}\right) $ goes to 0 as $\be$ goes to $+\8$. The importance of the formula is that, either $2\gamma>\al_{(j-1)p+1}\frth$ and the second term goes to $+\8$, or the possible convergence occurs at the sub exponential scale. 

In particular, it will show that only the components with $\alpha_{(j-1)p+1}$ sufficiently small can have weight as  $\be$ goes to $+\8$. 

%
%

In view to prove Proposition \ref{prop-value-mubeoj}, let us first start with some technical lemmas. 
\begin{lemma}
\label{lem-Fjtauj}
For every $j\neq j'$, $\disp \tau_{j}(1+F(\CP(\be),\al_{p(j-1)}\be,\ldots, \al_{jp}\be))=\tau_{j'}(1+F(\CP(\be),\al_{p(j'-1)}\be,\ldots, \al_{j'p}\be))$.
\end{lemma}
\begin{proof}

For a 
By Lemma \ref{lem2-aneis} $H_{\be}$ is constant on rings. Setting $\disp u_{1j}^{n}=\underbrace{u_{1j}\ldots u_{1j}}_{n\text{ times}}$ we get 

\begin{equation}
\label{equ-H-m-induit}
H_{\be}(u_{1j}^{n}*)=e^{-\CP(\be)}H_{\be}(u_{1j}^{n+1}*)\left(e^{-\al_{(j-1)p+1}\be\theta^{n+1}}+\ldots+e^{-\al_{jp}\be\theta^{n+1}}\right)+e^{-\CP(\be)}\tau_{j}.
\end{equation}
We set $\rho(j,n+1,\be):=\disp e^{-\al_{(j-1)p+1}\be\theta^{n+1}}+\ldots+e^{-\al_{jp}\be\theta^{n+1}}$. Note that 
$$F(\CP(\be),\al_{p(j-1)+1}\be,\ldots, \al_{jp}\be):=\sum_{k=1}^{+\8}e^{-k\CP(\be)}\prod_{i=1}^{k}\rho(j,i,\be),$$
and remember $\disp F_{l}(\CP(\be),\al_{p(j-1)+1}\be,\ldots, \al_{jp}\be):=\sum_{k=1}^{l}e^{-k\CP(\be)}\prod_{i=1}^{k}\rho(j,i,\be)$
Then, multiplying both sides of Equation \eqref{equ-H-m-induit} by $e^{-\CP(\be)}\rho(j,n,\be)$ and adding $e^{-\CP(\be)}\tau_{j}$ we get 
$$H_{\be}(u_{1j}^{n-1}*)=e^{-2\CP(\be)}\rho(j,n,\be)\rho(j,n+1,\be)H_{\be}(u_{1j}^{n+1}*)+e^{-\CP(\be)}\tau_{j}(1+e^{-\CP(\be)}\rho(j,n,\be)).$$

\medskip
We get, by induction, a relation between $H_{\be}(u_{1j}^{n}*)$, $n$ and $H_{\be}(u_{1j}*)$. Now, 
remember $\disp H_{\be}(u)=e^{-\CP(\be)}\rho(j,1,\be)+e^{-\CP(\be)}\tau_{j}$, and we finally get
\begin{equation}
\label{equ-Hu-tauj}
H_{\be}(u)=H_{\be}(u_{1j}^{n}*)\left[e^{-n\CP(\be)}\prod_{i=1}^{n-1}\rho(j,i,\be)\right]+e^{-\CP(\be)}\tau_{j}(1+F_{n}(\CP(\be),\al_{p(j-1)}\be,\ldots, \al_{jp}\be)).
\end{equation}
Now, let $n$ goes to $+\8$ in \eqref{equ-Hu-tauj}. The term $H_{\be}(u^{n}_{1j}*)$ converges to $\disp\frac{\tau_{j}}{e^{\CP(\be)}-p}$ (see Lemma \ref{lem-hb-const-aubry}) and the term $\disp e^{-n\CP(\be)}\prod_{i=1}^{n-1}\rho(j,i,\be)$ is the general term of a converging series, thus goes to $0$. Then we get 
$$H_{\be}(u)=e^{-\CP(\be)}\tau_{j}(1+F(\CP(\be),\al_{p(j-1)}\be,\ldots, \al_{jp}\be)).$$
This holds for any $j$.
\end{proof}

\bigskip

\begin{lemma}
\label{lem-Hbe-length-n}
For any $j$ and for any integer $n$, 
$$H_{\be}(u_{1j}^{n}*)=e^{-\CP(\be)}\tau_{j}(1+F(\CP(\be),\al_{p(j-1)}\be\theta^{n},\ldots, \al_{jp}\be\theta^{n})).$$
\end{lemma}
\begin{proof}
Equation \eqref{equ-H-m-induit} also yields 
$$H_{\be}(u_{1j}^{n}*)=e^{-2\CP(\be)}\rho(j,n+1,\be)\rho(j,n+2,\be)H_{\be}(u_{1j}^{n+2}*)+e^{-\CP(\be)}\tau_{j}(1+e^{-\CP(\be)}\rho(j,n+1,\be)).$$
By induction we get 
\begin{eqnarray*}
H_{\be}(u_{1j}^{n}*)&=&\left[e^{-i\CP(\be)}\prod_{l=1}^{i}\rho(j,n+l,\be)\right]H_{\be}(u_{1j}^{n+i}*)+e^{-\CP(\be)}\tau_{j}(1+\sum_{l=1}^{i}e^{-\CP(\be)}\prod_{r=1}^{l}\rho(j,n+r,\be))\\
&=& \left[e^{-i\CP(\be)}\prod_{l=1}^{i}\rho(j,n+l,\be)\right]H_{\be}(u_{1j}^{n+i}*)+e^{-\CP(\be)}\tau_{j}(1+F_{i}(\CP(\be),\al_{p(j-1)}\be\theta^{n},\ldots, \al_{jp}\be\theta^{n})).
\end{eqnarray*}
As above,  as $i$ goes to $+\8$, $H_{\be}(u^{n+i}_{1j}*)$ converges to $\disp\frac{\tau_{j}}{e^{\CP(\be)}-p}$ (see Lemma \ref{lem-hb-const-aubry}) and the term $\disp e^{-i\CP(\be)}\prod_{l=1}^{i}\rho(j,l,\be)$ is the general term of a converging series, thus goes to $0$. 
\end{proof}

\begin{lemma}
\label{lem-form-mum} 
$\disp \mu(O_{j})=e^{-\CP(\be)}\tau_{j}\sum_{l=1}^{+\8}le^{-l\CP(\be)}\prod_{i=1}^{l}\left(e^{-\al_{(j-1)p+1}\theta^{i}\be}+\ldots+e^{-\al_{jp}\theta^{i}\be}\right)(1-\nu_{\be}(O_{j}))$.
\end{lemma}
\begin{proof}
We pick some $j$. In the following $m$ is a generic $\S_{j}$-admissible word with finite length.

\begin{eqnarray}
\sum_{m,|m|=l}\mu_{\be}[m*]&=&\sum_{m,|m|=l}H_{\be}(m*)\nu_{\be}(m*)\nonumber\\
&=& \sum_{m,|m|=l}H_{\be}(m*)e^{-lP+\be.S_{l}(A)(m*)}(1-\nu_{\be}(O_{j}).\nonumber\\
&=& H_{\be}(u_{1j}^{l}*)e^{-l\CP(\be)}\prod_{i=1}^{l}\left(e^{-\al_{(j-1)p+1}\be\theta^{i}}+\ldots+e^{-\al_{jp}\be\theta^{i}}\right)(1-\nu_{\be}(O_{j}),\label{equ1-mubering-l}
\end{eqnarray}

Equality \eqref{equ1-mubering-l} and Lemma \ref{lem-Hbe-length-n} yield
\begin{eqnarray*}
\sum_{m,|m|=l}\mu_{\be}[m*]
&=&H_{\be}(u_{1j}^{l}*)e^{-l\CP(\be)}\prod_{i=1}^{l}\left(e^{-\al_{(j-1)p+1}\be\theta^{i}}+\ldots+e^{-\al_{jp}\be\theta^{i}}\right)(1-\nu_{\be}(O_{j})\\
&=& e^{-\CP(\be)}\tau_{j}\left(F(\CP(\be),\al_{(j-1)p+1}\be,\ldots, \al_{jp}\be)-F_{l-1}(\CP(\be),\al_{(j-1)p+1}\be,\ldots, \al_{jp}\be)\right)(1-\nu_{\be}(O_{j}))
\end{eqnarray*}

Now, it is an usual exercise that the sum of the tail of a series of positive terms $u_{n}$ is equal to $\sum_{n}nu_{n}$. 
\end{proof}

\begin{proof}[Proof of Proposition \ref{prop-value-mubeoj}]
We split the sum in the formula of Lemma \ref{lem-form-mum} in two pieces, the one for $l<n(\be)$ and the one for $l\ge n(\be)$.

For the part for $l<n(\be)$, we use Inequality \eqref{equ-majofine-Fnbe}, and 
$$\sum_{l=1}^{n(\be)-1}le^{-l\CP(\be)}\prod_{i=1}^{l}\left(e^{-\al_{(j-1)p+1}\theta^{i}\be}+\ldots+e^{-\al_{jp}\theta^{i}\be}\right)\le n(\be)F_{n(\be)-1}(\CP(\be),\al_{(j-1)p+1}\be,\ldots,,\al_{jp}\be).$$

For the sum for $l\ge n(\be)$, we use Equality \eqref{equ-estimF-1}. We have to ``update'' it and replace $e^{-ng(\be)e^{-\gamma\be}}$ by $ne^{-ng(\be)e^{-\gamma\be}}$. In other words, we are computing the formal power series $\disp\sum_{n\ge n(\be)}nx^{n}$ with $x=\disp e^{-g(\be)e^{-\gamma.\be}}$.  It is the derivative of the power series $\disp\sum_{n\ge n(\be)}x^{n}$. 
Hence we get
$$\sum_{n\ge n(\be)}nx^{n}=n(\be)\frac{x^{n(\be)-1}}{(1-x)^{2}}(1+x(1-\frac1{n(\be)})).$$

Again, $e^{-n(\be)g(\be)e^{-\gamma.\be}}$ goes to $1$ as $\be$ goes to $+\8$ and $1-e^{-g(\be)e^{-\gamma.\be}}$ behaves like $g(\be)e^{-\gamma.\be}$. 

\end{proof}

\subsection{Selection}
For simplicity we set\footnote{these are different from the truncated sums defined above.} $F_{j}:=F(\CP(\be),\al_{(j-1)p+1}\be,\ldots, \al_{jp}\be)$, $\eta_{ij}:=\eta_{(j-1)p+i}-\eta_{(j-1)p+1}$ and $I_{j}:=\frac1{\log\theta}I(\eta_{2j},\ldots, \eta_{pj})$. We remind that $o_{\8}(\be)$ means a function going to 0 as $\be$ goes to $+\8$. 

\subsubsection{The case $\gamma<\disp\frac{\al_{1}+\al_{p+1}}2\frth$}
This corresponds to zone $Z_{2}$ (see Figure \ref{fig-mu}). We emphasize that, if $j$ is such that $2\gamma< \al_{(j-1)p+1}\frth$, then Propositions \ref{prop-value-mubeoj} and \ref{prop-equil-F} show that $\mu_{\be}(O_{j})$ behaves like $\tau_{j}o_{\8}(\be)$. 
\begin{lemma}
\label{lem-mubeoj-jgrand}
Under the assumption $2\gamma\le \al_{(j-1)p+1}\frth$, $\mu_{\be}(O_{j})$ goes to 0 exponentially fast as $\be$ goes to $+\8$. 
\end{lemma}
\begin{proof}
Let us assume that $j$ is such that $2\gamma< \al_{(j-1)p+1}\frth$. 

\begin{eqnarray}
\frac{\mu_{\be}(O_{1})}{\mu_{\be}(O_{j})}&=& \frac{\tau_{1}}{(1+F_{1})\tau_{j}}\wh\kappa(\be)\frac1{\be^{r}(g(\be))^{2}}e^{(2\gamma-\al_{1}\frth)\be}e^{-I_{1}}\frac1{o_{\8}(\be)}\nonumber\\
&=& \frac{1+F_{j}}{(1+F_{1})^{2}}\wh\kappa(\be)\frac1{\be^{r}(g(\be))^{2}}e^{(2\gamma-\al_{1}\frth)\be}e^{-I_{1}}\frac1{o_{\8}(\be)}\nonumber\\
&& \text{(by Lemma \ref{lem-Fjtauj} )}\nonumber\\
&=& \phi(\be)\frac{e^{(2\gamma-\al_{1}\frth)\be}}{o_{\8}(\be)e^{2(\gamma-\al_{1}\frth)\be}}\nonumber\\
&& \text{ by Prop. \ref{prop-equil-F} and for some sub exponential function $\phi$}\nonumber\\
&=& \phi(\be)\frac{e^{\al_{1}\frth\be}}{o_{\8}(\be)}.\label{equ-rapportmube}
\end{eqnarray}

This shows that $\mu_{\be}(O_{1})$ is exponentially bigger (in $\be$) than $\mu_{\be}(O_{j})$, and thus $\mu_{\be}(O_{j})$ goes exponentially fast to $0$. 

The same holds if we only assume $2\gamma\le \al_{(j-1)p+1}\frth$, because for the equality case, we have just to replace the term $o_{\8}(\be)$ by some sub exponential quantity (which does not necessarily goes to 0). However, this does not eliminate  the exponential ratio in the computation. 

\end{proof}

\begin{remark}
\label{rem-measuO2}
Furthermore, this shows that $\mu_{\be}(O_{2})$ can have a positive accumulation point only if $2\gamma>\al_{p+1}\frth$. 
$\blacksquare$\end{remark}

\begin{lemma}
\label{lem-O2Oj}
Assume $2\gamma>\al_{p+1}\frth$. Then,  for every $j>2$, $\disp\lim_{\be\to+\8}\mu_{\be}(O_{j})=0$.
\end{lemma}
\begin{proof}
We copy the previous computation with $O_{2}$ instead of $O_{1}$. Note that $F_{2}$ goes to $0$ as $\be$ goes to $+\8$. We first assume $2\gamma\le\al_{(j-1)p+1}\frth$

Then we have 
\begin{eqnarray*}
\frac{\mu_{\be}(O_{2})}{\mu_{\be}(O_{j})}&=& \frac{\tau_{2}}{\tau_{j}}\frac{1+F_{j}}{1+F_{2}}\frac{\wh\kappa(\be)e^{(2\gamma-\al_{p+1}\frth)\be}}{o_{\8}(\be)},
\end{eqnarray*}
where $o_{\8}(\be)$ may be replaced by a sub exponential function (if $2\gamma=\al_{(j-1)p+1}\frth$). 
Note that $\disp \frac{\tau_{2}}{\tau_{j}}\to1$ and $\disp \frac{1+F_{j}}{1+F_{2}}\to 1$ as $\be\to+\8$. 
Then, $\mu_{\be}(O_{j})$ is exponentially smaller than $\mu_{\be}(O_{2})\le 1$.

If $2\gamma>\al_{(j-1)p+1}\frth$, then $\disp \frac{\mu_{\be}(O_{2})}{\mu_{\be}(O_{j})}$ is--- up to a sub-exponential multiplicative ratio--- equal to $\disp \frac{e^{(2\gamma-\al_{p+1}\frth)\be}}{e^{(2\gamma-\al_{(j-1)p+1}\frth)\be}}=e^{(\al_{(j-1)p+1}-\al_{p+1})\frth\be}$. Now, we remind 
$$\al_{p+1}<\al_{(j-1)p+1}.$$
\end{proof}

\begin{lemma}
\label{lem-O1-win-gamma}
If $\gamma<\frac{\al_{1}+\al_{p+1}}2\frth$, then $\disp\lim_{\be\to+\8}\mu_{\be}(O_{1})=1$.
\end{lemma}
\begin{proof}
The result holds if $2\gamma\le \al_{p+1}\frth$. Let us thus assume $2\gamma>\al_{p+1}\frth$. Then Equation \eqref{equ-rapportmube} is still valid, provide that we replace $o_{\8}(\be)$ by $\disp e^{-I_{2}}e^{(2\gamma-\al_{p+1}\frth)\be}$ (following Proposition \ref{prop-value-mubeoj}). Hence we get 
$$\frac{\mu_{\be}(O_{1})}{\mu_{\be}(O_{2})}=\phi(\be)e^{((\al_{1}+\al_{p+1})\frth-2\gamma)\be},$$
for some sub-exponential function $\phi(\be)$. Our assumption in the Lemma means that this last quantity diverges exponentially fast to $+\8$ as $\be$ goes to $+\8$, which means that $\mu_{\be}(O_{1})$ is exponentially bigger than $\mu_{\be}(O_{2})$ as $\be$ goes to $+\8$. Hence, $\mu_{\be}(O_{2})$ goes to $0$. Lemma \ref{lem-O2Oj} shows that $\mu_{\be}(O_{j})$ goes to $0$ for every $j\neq 1$, thus $\mu_{\be}(O_{1})$ goes to 1.  
\end{proof}

\subsubsection{The case $\gamma=\disp\frac{\al_{1}+\al_{p+1}}2\frth<\min(\al_{1}\frth+\al,\al_{1}\frth+\al_{p+1}\theta)$}
This corresponds to zone $Z_{1}$ (see Figure \ref{fig-mu}). 
We emphasize that in that case, 
$$2\gamma>\al_{p+1}\frth$$
always holds. Therefore only $O_{1}$ and $O_{2}$ can have weight for $\be\to+\8$. 
Moreover, $\disp\gamma=\frac{\al_{1}+\al_{p+1}}2\frth$ implies 
$$\frac{\al_{1}+\al_{p+1}}2\frth<\al_{1}\frth+\al_{p+1}\theta,$$
which yields 
$$\disp\gamma-\al_{p+1}\frth=\frac{\al_{1}-\al_{p+1}}2\frth>-\al_{p+1}\theta.$$
Then, Proposition \ref{prop-equil-F} shows 
\begin{equation}
\label{equ-valueF2}
F_{2}=\frac{1}{\be^{r}g(\be)}e^{-I_{2}}e^{(\gamma-\al_{p+1}\frth)\be}(1+o_{\8}(\be)). 
\end{equation}
It  also shows that for every $j>2$, $F_{j}$ is of order $\max(e^{-\al_{(j-1)p+1}\be}, e^{(\gamma-\al_{(j-1)p+1}\frth)\be})$. It is thus exponentially smaller than $F_2$. 

\begin{lemma}
\label{lem-F1F2}
Under these hypothesis, $\disp\lim_{\be\to+\8}F_{1}F_{2}=p^{2}$ and 
$\disp g(\be)=\frac1{p\be^{r}}e^{-\frac{I_{1}+I_{2}}2}(1+o_{\8}(\be))$.
\end{lemma}
\begin{proof}
Equation \eqref{equ-nube-mot} can be rewritten under the form:
$$\nu_{\be}(O_{j})=e^{-\CP(\be)}\sum_{i\neq j}F_{i}\nu_{\be}(O_{i})+e^{-\CP(\be)}\nu_{\be}([u]).$$
This yields a linear system
$$\left(\begin{array}{cccc}1 & -e^{-\CP(\be)}F_1 & \ldots & -e^{-\CP(\be)}F_1 \\-e^{-\CP(\be)}F_2 & 1 &  & -e^{-\CP(\be)}F_2 \\\vdots &  & \ddots & \vdots \\-e^{-\CP(\be)}F_N & \ldots & -e^{-\CP(\be)}F_N & 1\end{array}\right)
\left(\begin{array}{c}\nu_\be(O_1) \\\nu_\be(O_2) \\\vdots \\\nu_\be(O_N)\end{array}\right)=\left(\begin{array}{c}e^{-\CP(\be)}\nu_\be([u])F_{1} \\e^{-\CP(\be)}\nu_\be([u])F_{2} \\\vdots \\e^{-\CP(\be)}\nu_\be([u]F_{N})\end{array}\right).
$$

We remind that $F_{j}$ goes to $0$ as $\be$ goes to $+\8$ for $j\ge 2$, and $F_{1}$ goes to $+\8$. 

We compute the dominating term of the determinant of the $N\times N$ matrix in the left hand side of the last equality. 
Developing this determinant with respect to the first row, we left it to the reader to check that the determinant is of the form 
$$\det(\be)=1-e^{-2\CP(\be)}F_{1}\left(\sum_{i=2}^{N}F_{i}\right)(1+o_{\8}(\be)).$$
Now we compute the cofactors for the terms of the first column. 
Again, we left it to the reader to check that the cofactor of the term in position $i,1$ is of the form 
$$\delta_{1,1}=1+o_{\8}(\be), \quad \delta_{i,1}=-e^{-\CP(\be)}F_{1}(1+o_{\8}(\be)).$$
Now, remind Equality \eqref{equ-value-nube-u} $\nu_{\be}([u])=\disp e^{-\CP(\be)-\al\be}$.

Therefore,    Equality \eqref{equ-valueF2} and the property $F_{j}<< F_{2}$  (for $j>2$) yield
\begin{equation}
\label{equ1-F1F2}
\nu_{\be}(O_{1})=\frac{F_{1}e^{-2\CP(\be)-\al}(1+o_{\8}(\be))}{1-e^{-2\CP(\be)}F_{1}F_{2}(1+o_{\8}(\be))}.
\end{equation}
We remind that with our values of the parameter, $\disp\gamma=\frac{\al_{1}+\al_{p+1}}2\frth<\al_{1}\frth+\al$. 
Now, $F_{1}$ behaves (at the exponential scale) like $e^{(\gamma-\al_{1}\frth)\be}<\al\be$. This shows that the numerator in \eqref{equ1-F1F2} goes to $0$ as $\be$ goes to $+\8$. Therefore the denominator also goes to 0 and the first part of the Lemma is proved as $\lim_{\be\to+\8}\CP(\be)=p$.

Let us now replace $F_{1}$ and $F_{2}$ by their values. Following Proposition \ref{prop-equil-F} and \eqref{equ-valueF2} we get 
\begin{eqnarray*}
F_{1}F_{2}&=& e^{-I_{1}-I_{2}}\frac1{(\be^{r}g(\be))^{2}}e^{(2\gamma-(\al_{1}+\al_{p+1})\frth)\be}(1+o_{\8}(\be))\\
&=& e^{-I_{1}-I_{2}}\frac1{(\be^{r}g(\be))^{2}}(1+o_{\8}(\be)).
\end{eqnarray*}
As $F_{1}F_{2}$ goes to $p^{2}$, we get $\disp g(\be)=\frac1{p\be^{r}}e^{-\frac{I_{1}+I_{2}}2}(1+o_{\8}(\be))$. 
\end{proof}

We can now finish the proof of Theorem 2--- item {\it(1)}. We remind that we get 
$$\mu_{\be}(O_{1})=e^{-\CP(\be)}\frac{\tau_{1}}{1+F_{1}}\frac{n(\be)}{\be^{r}(g(\be))^{2}}e^{-I_{1}}e^{(2\gamma-\al_{1}\frth)\be},$$
and 
$$\mu_{\be}(O_{2})=e^{-\CP(\be)}\frac{\tau_{2}}{1+F_{2}}\frac{n(\be)}{\be^{r}(g(\be))^{2}}e^{-I_{2}}e^{(2\gamma-\al_{p+1}\frth)\be}.$$
We also remind 
\begin{itemize}
\item $\tau_{1}(1+F_{1})=\tau_{2}(1+F_{2})$ (Lemma \ref{lem-Fjtauj})
\item $F_{2}\to_{\be\to+\8} 0$ and $F_{1}\to_{\be\to+\8} +\8$,
\item $\disp g(\be)=\frac1{p\be^{r}}e^{-\frac{I_{1}+I_{2}}2}(1+o_{\8}(\be))$ (Lemma \ref{lem-F1F2}),
\item $\gamma=\disp\frac{\al_{1}+\al_{p+1}}2\frth$.
\end{itemize}
Then we get 
\begin{eqnarray*}
\frac{\mu_{\be}(O_{1})}{\mu_{\be}(O_{2})}&=&\frac{\tau_{1}}{(1+F_{1})^{2}}\frac{1+F_{1}}{\tau_{2}}(1+F_{2})e^{I_{2}-I_{1}}e^{(\al_{p+1}-\al_{1})\frth\be}(1+o_{\8}(\be)\\
&=&\frac{e^{I_{2}-I_{1}}}{(1+F_{1})^{2}} e^{(\al_{p+1}-\al_{1})\frth\be}(1+o_{\8}(\be))\\
&=& e^{I_{2}-I_{1}}e^{2I_{1}}(\be^{r}g(\be))^{2}e^{(\al_{p+1}-\al_{1})\frth\be}e^{2(\al_{1}\frth-\gamma)\be}(1+o_{\8}(\be))\\
&=& \frac1{p^{2}}(1+o_{\8}(\be)). 
\end{eqnarray*}

\subsubsection{The case $\gamma=\disp\frac{\al_{1}+\al_{p+1}}2\frth=\al_{1}\frth+\al<\al_{1}\frth+\al_{p+1}\theta$}
This corresponds to zone $Z_{3}\setminus Z_{4}$ (see Figure \ref{fig-mu}).
The situation is very similar than the previous one. We rewrite Equality \eqref{equ1-F1F2}:
$$
\nu_{\be}(O_{1})=\frac{F_{1}e^{-2\CP(\be)-\al}(1+o_{\8}(\be))}{1-e^{-2\CP(\be)}F_{1}F_{2}(1+o_{\8}(\be))}.
$$
Again we claim that we get $\disp F_{1}F_{2}=e^{-I_{1}-I_{2}}\frac1{(\be^{r}g(\be))^{2}}(1+o_{\8}(\be))
$. 

Nevertheless, and contrarily to the previous case, the numerator is equal to 
$$e^{-2\CP(\be)-\al}F_{1}=e^{-I_{1}}\frac1{\be^{r}g(\be)}.$$
Let $L$ be any accumulation for $\disp\frac1{\be^{r}g(\be)}$. Then we get 
$$1=\frac{e^{-I_{1}}L}{p^{2}-e^{-I_{1}-I_{2}}L^{2}}.$$
Note that $L\ge 0$, then solving the equation we get only one non-negative solution. Hence

$$\lim_{\be\to+\8}\frac1{\be^{r}g(\be)}=\frac{\sqrt{4p^{2}e^{I_{1}-I_{2}}}-1}2e^{I_{2}}.$$
Now, copying what is done above we get 
\begin{eqnarray*}
\frac{\mu_{\be}(O_{1})}{\mu_{\be}(O_{2})}&=&\frac{\tau_{1}}{(1+F_{1})^{2}}\frac{1+F_{1}}{\tau_{2}}(1+F_{2})e^{I_{2}-I_{1}}e^{(\al_{p+1}-\al_{1})\frth\be}(1+o_{\8}(\be)\\
&=&\frac{e^{I_{2}-I_{1}}}{(1+F_{1})^{2}} e^{(\al_{p+1}-\al_{1})\frth\be}(1+o_{\8}(\be))\\
&=& e^{I_{2}-I_{1}}e^{2I_{1}}(\be^{r}g(\be))^{2}e^{(\al_{p+1}-\al_{1})\frth\be}e^{2(\al_{1}\frth-\gamma)\be}(1+o_{\8}(\be))\\
&=& \frac{4e^{I_{1}-I_{2}}}{\left(\sqrt{4p^{2}e^{I_{1}-I_{2}}+1}-1\right)^{2}}(1+o_{\8}(\be)). 
\end{eqnarray*}

We set 
\begin{equation}
\label{equ-def-rho-3}
\rho^{2}_{3}:=\frac{4e^{I_{1}-I_{2}}}{\left(\sqrt{4p^{2}e^{I_{1}-I_{2}}+1}-1\right)^{2}}.
\end{equation}

\subsubsection{The case $\gamma=\disp\frac{\al_{1}+\al_{p+1}}2\frth=\al_{1}\frth+\al_{p+1}\theta<\al_{1}\frth+\al$}
This corresponds to zone $Z_{4}\setminus Z_{3}$ (see Figure \ref{fig-mu}).

\begin{lemma}
\label{lem-F2-Z4-3}
The quantity $F_{2}$ is exponentially bigger than every $F_{j}$ for $j>2$.
\end{lemma}
\begin{proof}
We remind that $\al_{p+1}=\disp\frac{\al_{1}}{2\theta-1}$. 
Then, Proposition \ref{prop-equil-F} shows that $F_{2}$ behaves, at the exponential scale, like $\disp\max(e^{-\al_{p+1}\theta\be}, e^{(\gamma-\al_{p+1}\frth)\be})$ and these two quantities are equal. 

Now, for $j>2$, $F_{j}$ is lower than $\disp e^{-\al_{(j-1)p+1}\theta}$, thus exponentially smaller than $F_{2}$.
\end{proof}

Again, we rewrite Equality \eqref{equ1-F1F2}. Lemma \ref{lem-F2-Z4-3} shows that $F_{2}+F_{3}+\ldots$ is again of the form $F_{2}(1+o_{\8}(\be))$:
$$
\nu_{\be}(O_{1})=\frac{F_{1}e^{-2\CP(\be)-\al}(1+o_{\8}(\be))}{1-e^{-2\CP(\be)}F_{1}F_{2}(1+o_{\8}(\be))}.
$$

\begin{lemma}
\label{lem-lim-bergbe-Z4-3}
The quantity $F_{1}F_{2}$ goes to $p^{2}$ as $\be$ goes to $+\8$ and $\be^{r}g(\be)$ converges as $\be$ goes to $+\8$. The limit is denoted by $L$. 
\end{lemma}
\begin{proof}
Remember that $F_{1}$ behaves like $e^{(\gamma-\al_{1}\frth)\be}=e^{\al_{p+1}\frth\be}$, and our assumption yields that the numerator  in the last expression for $\nu_{\be}(O_{1})$ goes exponentially fast to $0$. Hence, the denominator also goes to $0$ and (again) $F_{1}F_{2}$ goes to $p^{2}$.

Let $L$ be any accumulation point for $\be^{r}g(\be)$ (in $\overline{\R}_{+}$). Proposition \ref{prop-equil-F} shows that $F_{1}$ behaves like $\disp e^{-I_{1}}\frac1Le^{(\gamma-\al_{1}\frth)\be}(1+o_{\8}(\be))$ and $F_{2}$ behaves like $\disp\max(1,e^{-I_{2}}\frac1L)e^{(\gamma-\al_{p+1}\frth)\be}(1+o_{\8}(\be))$. 

This yields equality $\disp F_{1}F_{2}= e^{-I_{1}}\frac1L.\max(1,e^{-I_{2}}\frac1L)(1+o_{\8}(\be)).
$ and doing $\be\to+\8$ we get 
\begin{equation}
\label{equ-limbergbe-Z4-3}
p^2= e^{-I_{1}}\frac1L.\max(1,e^{-I_{2}}\frac1L).
\end{equation} 
Now, the function $x\mapsto \disp e^{-I_{1}}x\max(1,e^{-I_{2}}x)$ is increasing. Thus there exists a unique $x$ such that its value is $p^{2}$. This proves that $\be^{r}g(\be)$ has unique accumulation point, thus converges. 
\end{proof}

Now, copying what is done above we get 
\begin{eqnarray*}
\frac{\mu_{\be}(O_{1})}{\mu_{\be}(O_{2})}&=&\frac{\tau_{1}}{(1+F_{1})^{2}}\frac{1+F_{1}}{\tau_{2}}(1+F_{2})e^{I_{2}-I_{1}}e^{(\al_{p+1}-\al_{1})\frth\be}(1+o_{\8}(\be)\\
&=&\frac{e^{I_{2}-I_{1}}}{(1+F_{1})^{2}} e^{(\al_{p+1}-\al_{1})\frth\be}(1+o_{\8}(\be))\\
&=& e^{I_{2}-I_{1}}e^{2I_{1}}(\be^{r}g(\be))^{2}e^{(\al_{p+1}-\al_{1})\frth\be}e^{2(\al_{1}\frth-\gamma)\be}(1+o_{\8}(\be))\\
&=& e^{I_{1}+I_{2}}L^{2}(1+o_{\8}(\be)). 
\end{eqnarray*}

We set 
\begin{equation}
\label{equ-def-rho-4}
\rho^{2}_{4}:=e^{I_{1}+I_{2}}L^{2}.
\end{equation}

\subsubsection{The case $\gamma=\disp\frac{\al_{1}+\al_{p+1}}2\frth=\al_{1}\frth+\al=\al_{1}\frth+\al_{p+1}\theta$}
This corresponds to zone $Z_{3}\cap Z_{4}$ (see Figure \ref{fig-mu}).

The situation is very similar than the previous one. Lemma \ref{lem-F2-Z4-3} still holds as we just used inequalities $\al_{p+1}<\al_{(j-1)p+1}$ for $j>2$. Hence, $\disp F_{1}F_{2}=e^{-I_{1}-I_{2}}\frac1{(\be^{r}g(\be))^{2}}(1+o_{\8}(\be))
$.

The main difference with the previous case is that writing Equality \eqref{equ1-F1F2}:
$$
\nu_{\be}(O_{1})=\frac{F_{1}e^{-2\CP(\be)-\al}(1+o_{\8}(\be))}{1-e^{-2\CP(\be)}F_{1}F_{2}(1+o_{\8}(\be))},
$$
the numerator does not necessarily goes to $0$. Namely it behaves like\footnote{note that $\gamma=\al_{1}\frth+\al$ and $F_{1}$ behaves like $\disp e^{-I_{1}}\frac1{\be^{r}g(\be)}e^{(\gamma-\al_{1}\frth)\be}$.} $\disp e^{-I_{1}}\frac1{\be^{r}g(\be)}$.

Nevertheless, Lemma \ref{lem-lim-bergbe-Z4-3} still holds. Indeed, Equality \eqref{equ-limbergbe-Z4-3} has just to be replaced by 
$$1=\frac{e^{-I_{1}}\frac1L}{p^{2}-e^{-I_{1}}\frac1L\max(1,e^{-I_{2}}\frac1L)}.$$
Again, the function $x\mapsto e^{-I_{1}}x(1+\max(1,e^{-I_{2}}x))$ is increasing and there is a unique value for which the function is equal to $p^{2}$.

Now, copying what is done above we get 
\begin{eqnarray*}
\frac{\mu_{\be}(O_{1})}{\mu_{\be}(O_{2})}&=&\frac{\tau_{1}}{(1+F_{1})^{2}}\frac{1+F_{1}}{\tau_{2}}(1+F_{2})e^{I_{2}-I_{1}}e^{(\al_{p+1}-\al_{1})\frth\be}(1+o_{\8}(\be)\\
&=&\frac{e^{I_{2}-I_{1}}}{(1+F_{1})^{2}} e^{(\al_{p+1}-\al_{1})\frth\be}(1+o_{\8}(\be))\\
&=& e^{I_{2}-I_{1}}e^{2I_{1}}(\be^{r}g(\be))^{2}e^{(\al_{p+1}-\al_{1})\frth\be}e^{2(\al_{1}\frth-\gamma)\be}(1+o_{\8}(\be))\\
&=& e^{I_{2}+I_{1}}L^{2}(1+o_{\8}(\be)). 
\end{eqnarray*}

We set 
\begin{equation}
\label{equ-def-rho-34}
\rho^{2}_{3\cap4}:=e^{I_{1}+I_{2}}L^{2}.
\end{equation}

This concludes the proof of Theorem 2. 

\section{Convergence to the subaction- proof of corollary 3}\label{sec-proof-coro}

In the proof of Proposition \ref{prop2-maxplus} we showed that only two basic loops can have a maximal weight (which value is $\gamma$). These two loops are $1\to 1$ or $1\to 2\to 1$. Following the Max-Plus formalism (see \cite{BCOQ}), in both cases there is a unique \emph{maximal strongly connected subgraph} (m.s.c.s. in abridge way) which is either $1\to 1$ or $1\to2\to1$. 

Now, Theorem 3.101 in \cite{BCOQ} shows that, in both cases, there is a unique eigenvector for $M$, up to an additive constant (added to all the coordinates), which is given by the first column of the associated matrix $M^{+}:=e\oplus M\oplus M^{2}\ldots$ (where $\oplus$ is the sum for the Max-plus formalism, and $M^{n}$ is computed for the product of the Max-Plus formalims). 

In other word, there is a unique calibrated subaction up to an additive constant. Consequently, a subaction is entirely determined by its value on one of any $\S_{j}$'s. 

Now, we have:
\begin{lemma}
\label{lem-VS1}
For every $x$ in $S_{1}$, $\disp\lim_{\be\to+\8}\frac1\be\log H_{\be}(x)=0$.
\end{lemma}
\begin{proof}
The proof is done by contradiction. Assume that $\delta\neq0$ is an accumulation point for $\disp\frac1\be\log H_{\be}(x)$ (for $x$ in $\S_{1}$). Let consider any accumulation point $V$ such that $V(x)=\delta$ (this is always possible up to consider a subsequence of $\be$'s). 

Then, $V$ is H\"older continuous (all the $\disp\frac1\be\log H_{\be}$ are equip-continuous with a upper bounded H\"older norm), and $|V|>\frac\delta2$ on some neighborhood of $x$. Let us consider some cylinder $C$ in $\S_{1}$, such that for every $x'\in C$, $|V(x')-V(x)|<\frac\delta2$. 

By Theorem 2, $\mu_{\be}(C')$ converges to a positive value as $\be$ goes to $+\8$ (each cylinder in $\S_{1}$ has positive $\mu_{top,1}$-measure). Similarly, by Theorem 1, $\nu_{\be}(C')$ converges to a positive value as $\be$ goes to $+\8$. 
Now, 
$$d\mu_{\be}=H_{\be}d\nu_{\be},$$
which yields that $\mu_{\be}(C')$ is of order $\nu_{\be}(C')e^{\be\frac\delta2}$ as $\be$ goes to $+\8$. Convergence of $\nu_{\be}(C')$ and $\mu_{\be}(C')$ and $\delta\neq0$ yield a contradiction. 
\end{proof}

Lemma \ref{lem-VS1} shows that any accumulation point for $\disp\frac1\be\log H_{\be}$   is the unique subaction whose value is 0 on $\S_{1}$, thus it converges.

\end{document}